\documentclass[12pt]{amsart}

\usepackage{amsmath,palatino}
\usepackage{amsthm}
\usepackage{enumerate}
\usepackage{amssymb}
\usepackage{graphicx}
\usepackage[utf8]{inputenc}
\usepackage[all]{xy}

\theoremstyle{plain}
\newtheorem{lemma}{Lemma}[section]
\newtheorem{theorem}[lemma]{Theorem}
\newtheorem{corollary}[lemma]{Corollary}
\newtheorem{proposition}[lemma]{Proposition}
\newtheorem{definition}[lemma]{Definition}
\newtheorem*{proposition*}{Proposition}
\newtheorem*{definition*}{Definition}

\newtheorem{notation}[lemma]{Notation}

\newtheorem{remark}[lemma]{Remark}

 \newcommand{\CC}{\mathrm{C}}
 
 \newcommand{\M}{\mathrm{M}}

 \newcommand{\N}{\mathbb{N}}
 
 \newcommand{\R}{\mathbb{R}}
 \newcommand{\C}{\mathbb{C}}
 \newcommand{\Z}{\mathbb{Z}}
 \newcommand{\K}{\mathcal{K}}

 \newcommand{\Her}{\mathrm{Her}}

\DeclareMathOperator{\Cu}{Cu}
\DeclareMathOperator{\Lsc}{Lsc}
\DeclareMathOperator{\bd}{bd}

\begin{document}

\title{\textbf{Pullbacks, $\mathbf{C(X)}$-algebras, and their Cuntz semigroup}}
\author{Ramon Antoine}\email{ramon@mat.uab.cat}
\author{Francesc Perera}\email{perera@mat.uab.cat}
\author{Luis Santiago}\email{santiago@mat.uab.cat}
\date{\today}

\begin{abstract}
In this paper we analyse the structure of the Cuntz semigroup of certain $\CC(X)$-algebras, for compact spaces of low dimension, that have no $\mathrm{K}_1$-obstruction in their fibres in a strong sense. The techniques developed yield computations of the Cuntz semigroup of some surjective pullbacks of C$^*$-algebras. As a consequence, this allows us to give a complete description, in terms of semigroup valued lower semicontinuous functions, of the Cuntz semigroup of $\CC(X,A)$, where $A$ is a not necessarily simple C$^*$-algebra of stable rank one and vanishing $\mathrm{K}_1$ for each closed, two sided ideal. We apply our results to study a variety of examples.
\end{abstract}

\maketitle


\section*{Introduction}

To any C$^*$-algebra $A$, one can attach an ordered semigroup $\Cu(A)$, the Cuntz semigroup of $A$.
It was originally devised by Cuntz in \cite{cu}, and can be constructed using suitable equivalence classes
of positive elements in the stabilisation of $A$, in a similar way as the projection semigroup is built
out of the Murray-von Neumann equivalence.  
Coward, Elliott and Ivanescu proved in \cite{CEI} that the order relation in $\Cu(A)$ has additional properties and 
a new category Cu was defined for ordered semigroups 
with this structure. This category was shown to be closed under sequential inductive limits, and it was furthermore proved that
the assignment $A\mapsto \Cu(A)$, from the category of C$^*$-algebras to the category Cu, is functorial and (sequentially) continuous.

The study of the Cuntz semigroup has had a resurgence in recent years, mainly due to its impact in Elliott's classification
program. Notably, one of its order properties is the key to distinguish two non-isomorphic
C$^*$-algebras that agree on the Elliott invariant and several possible extensions of it; see \cite{toms}. 
For well behaved simple algebras, this semigroup can be recovered from the classical Elliott invariant (see \cite{bpt}, \cite{bt}), and in the non-simple case it has already been
used to prove actual classification results (see, e.g. \cite{elliottciuperca}, \cite{Ciuperca-Elliott-Santiago}, \cite{robertnccw}, \cite{Robert-Santiago}, \cite{santiago}). 

However, due to its complexity, this semigroup becomes an object very difficult to 
describe. Already in the commutative setting, it requires to understand the isomorphism classes of fibre bundles over
the spectrum $X$ of the algebra. In cases where there are no cohomological obstructions,
a description via point evaluations has been obtained in terms of (extended) integer valued lower semicontinuous functions on $X$
(see \cite{leonel}, \cite{leoneltikuisis}). A natural class to consider consists of those algebras that have the form $\CC_0(X,A)$, for a
locally compact Hausdorff space $X$. When $A$ is a unital, simple, non type
I ASH-algebra with slow dimension growth, the case $\CC_0(X,A)$ has been studied in \cite{Tikuisis}.
The description of $\Cu(\CC_0(X,A))$ is given in terms of pairs of certain projection valued functions and semigroup
valued lower semicontinuous functions (see below).

In this paper we study the Cuntz semigroup of $\CC(X)$-algebras $A$ when $X$ is a second countable, compact Hausdorff
space of dimension at most one. We also assume that each fibre of $A$ has stable rank one and vanishing $\mathrm{K}_1$ for
every closed, two sided ideal. Our approach is based on describing Cuntz equivalence classes by means of the corresponding
classes in the fibres, thus seeking to recover global information from local data. We do so by analysing the natural map
\[
\Cu(A)\longrightarrow\prod_{x\in X}\Cu(A(x))\,,[a]\longmapsto ([a(x)])_{x\in X}\,.
\]
Of particular interest will be the algebras 
of the form $\CC_0(X,A)$ for a not necessarily simple algebra $A$, as in this case the range of the previous map can be completely identified. 
This is achieved in Theorem \ref{thm:ese}. The strategy 
combines a number of ingredients, each of which may well have independent interest. We discuss them below.

In Section 2, we are exclusively concerned with $X=[0,1]$ and prove in Theorem \ref{th: C[0,1]-algebras} that the map above is an order-embedding for a $\CC([0,1])$-algebra whose fibres have the said conditions. If, further, the algebra 
has the form $\CC([0,1],A)$, we show that the range of the map can be described as the
semigroup of lower semicontinuous functions on $X$ with values in $\Cu(A)$, denoted by $\Lsc(X,\Cu(A))$
(Corollary \ref{cor:intervalo}). To do this,  we need to show that this semigroup belongs to the category Cu,
 and this requires a deeper analysis of general 
 semigroups of the form $\Lsc(X,M)$ with $M$ a semigroup in Cu. In order not to interrupt the flow of the paper, we postpone this discussion until 
 Section 5, where in fact we prove that, for any finite-dimensional second countable, compact Hausdorff space $X$ and 
 any separable C$^*$-algebra $A$, the semigroup $\Lsc(X,\Cu(A))$ belongs to Cu (see Theorem \ref{thm:lscinCu}). 
 
Pullbacks are the main theme in Section 3 as they provide us a way to deal with more general spaces.
We consider surjective pullbacks of the form $B\oplus_{A(Y)}A$, where $A$ is a $\CC(X)$-algebra, $Y$ is a
closed subset of $X$ and $B$ is any C$^*$-algebra (where the maps are given by the natural projection $A\to A(Y)$ and a $*$-homomorphism $B\to A(Y)$). Any such pullback gives rise to a diagram of the corresponding Cuntz semigroups.
We relate, in Theorems \ref{thm:injectivepullback1}, \ref{thm:injectivepullback} and \ref{thm:surjectivepullback}, the pullback in the category of ordered semigroups with
the Cuntz semigroup of the C$^*$-algebra pullback. The methods developed allow us to prove, in Theorem \ref{thm:ese}, that the Cuntz semigroup of $\CC(X,A)$ is order-isomorphic to $\Lsc(X,\Cu(A))$, where $X$ has dimension at most one, $A$ has stable rank one and vanishing $\mathrm{K}_1$ for each closed, two sided ideal. If, further, $A$ is an AF-algebra, the same result is available for spaces of dimension at most 2 and vanishing second \v Cech cohomology group (Corollary \ref{cor: pullbackdimtwo}). A key ingredient in the proofs is the continuity of the functors $\Lsc(X,-)$ and $\Lsc(-,\Cu(A))$, also proved to hold in Section 5 (Proposition \ref{prop:limits}). This description yields, as a consequence, a cancellation result, namely that $\Cu(\CC(X,A))$ is order-cancellative with respect to the relation $\ll$ (see below).

In Section 4 we discuss applications of the previous results in some computations of $\Cu(A)$.
Mainly, combining the results above, we can give a description of the Cuntz semigroup for algebras obtained by a successive pullback construction. This includes
one dimensional rsh-algebras and, more specifically, one dimensional non commutative CW-complexes. 
We also offer a computation of the Cuntz semigroup of dimension drop algebras
over the interval and over certain two dimensional spaces, as well as a description of
the Cuntz semigroup of the mapping torus of an algebra $A$.



\section{Preliminary results and definitions}

\subsection{Cuntz Semigroup}

Let $A$ be a C$^*$-algebra. Recall that a positive element $a$ is said to be \emph{Cuntz subequivalent} 
to $b\in A_+$, written $a\precsim b$, if there exists a sequence $(x_n)$ in $A$ such that $x_nbx_n^*\to a$. 
This defines a preorder in $A_+$ and we say that $a$ is \emph{Cuntz equivalent} to $b$, $a\sim b$, if
$a\precsim b$ and $b\precsim a$.

\begin{proposition}[{\cite[Proposition 2.4]{Rorfunct},\cite[Proposition 2.6]{KirchRor}}]\label{lem:extror}
 Let $A$ be a C$^*$-algebra, and $a,b\in A_+$. The following are equivalent:
\begin{enumerate}[{\rm (i)}]
 \item $a \precsim b$. 
\item For all $\epsilon>0$, $(a-\epsilon)_+\precsim b$.
\item For all $\epsilon>0$, there exists $\delta>0$ such that $(a-\epsilon)_+\precsim (b-\delta)_+$.
\end{enumerate}
Furthermore, if $A$ is stable, this conditions are equivalent to
\begin{enumerate}[{\rm (i)}]\setcounter{enumi}{3}
 \item For every $\epsilon>0$ there is a unitary $u\in
\mathrm{U}(A^\sim)$ such that $u(a-\epsilon)_+u^*\in\mathrm{Her}(b)$.
\end{enumerate}
\end{proposition}

\begin{proof}
The proof of the equivalence between (i) and (iv) is essentially that of \cite[Proposition 2.4 (v)]{Rorfunct},
where it is stated for algebras with stable rank one. We briefly sketch the necessary modifications to extend it to 
the case of stable algebras.

Given $\epsilon>0$, write $(a-\epsilon/2)_+=zz^*$ with $z^*z\in\mathrm{Her}(b)$. 
By \cite[Lemma 4.8]{algrad}, we know that $A\subset\overline{\mathrm{GL}(A^\sim)}$. Therefore
$\mathrm{dist}(z^*,\mathrm{GL}(A^\sim))=0$, so \cite[Corollary 8]{ped} applies to find $u\in\mathrm{U}(A^\sim)$ with
\[
u(a-\epsilon)_+u^*=v(a-\epsilon)_+v^*=(z^*z-\epsilon/2)_+\in\mathrm{Her}(b)\,,
\]
where $v$ is the partial isometry in the polar decomposition of $z^*$.
\end{proof}

The \emph{Cuntz semigroup} of $A$ is defined as the set of Cuntz equivalence classes in the stabilized algebra, $(A\otimes\mathcal K)_+/\sim$
and is denoted by $\Cu(A)$. Equip $\Cu(A)$ with the order induced by Cuntz subequivalence and addition
given by 
\[[ a] +[ b] =
\big[ \begin{pmatrix}a & 0 \\ 0 & b\end{pmatrix}\big],\]
so that it becomes an ordered Abelian semigroup with $[ 0 ]$ as its least element (the positive element inside the brackets in the right side of the equation above is identified with its image in $A\otimes\K$ by any isomorphism of $\M_2(A\otimes \K)$ with $A\otimes \K$ induced by an isomorphism of $\K$ and $\M_2(\K)$).

Recall that, in an ordered semigroup $M$, $a$ is said to be \emph{compactly contained} in $b$, denoted $a\ll b$,
if whenever $b\leq\sup_n c_n$ for some increasing sequence $(c_n)$ with supremum in $M$, implies there exists $n_0$ such 
that $a\leq c_{n_0}$. A sequence $(a_n)$ such that $a_n\ll a_{n+1}$ is
said to be \emph{rapidly increasing}.  The following theorem summarizes the structure of $\Cu(A)$.

\begin{theorem}[\cite{CEI}]\label{th:CEI}
 Let $A$ be a C$^*$-algebra. Then:
\begin{enumerate}[{\rm (i)}]
\item $\Cu(A)$ is closed under suprema of increasing sequences.
\item Any element in $\Cu(A)$ is the supremum of a rapidly increasing sequence. 
\item The operation of taking suprema and $\ll$ are compatible with addition.
\end{enumerate}
 \end{theorem}

The conditions (i)-(iii) of the last theorem define a category of ordered semigroups of positive elements, denoted by $\Cu$, 
which is closed under countable inductive limits, and such that $\Cu(-\otimes \mathcal K)$ defines a sequentially continuous
functor from the category of C$^*$-algebras to the category $\Cu$ (see \cite{CEI}). It can be shown that the rapidly increasing sequence in (ii) for a positive element 
$a\in A\otimes \mathcal K$, can be chosen as $([ (a-\frac{1}{n}))_+])$.

Let $M$ be a semigroup in the category $\Cu$.  Endow $M$ with the $\omega$-Scott topology, that is, 
the topology generated by the open sets $a^{\ll}:=\{c\in M\ \mid a\ll c\}$ where $a\in M$ (see \cite{scottetal}).
Adopting the terminology of \cite{scottetal}, we will say that an object $M$ in Cu is \emph{countably based} if there is a
countable subset $X$ in $M$ such that every element of $M$ is the supremum of a rapidly increasing sequence of elements coming from $X$. 
If $M$ is countably based, then $M$ satisfies the second axiom of countability as a topological space equipped with the Scott topology
(see, e.g. \cite[Theorem III-4.5]{scottetal}).


\begin{lemma}
\label{lem:countbase}
Let $A$ be a separable C$^*$-algebra. Then $\mathrm{Cu}(A)$ is countably based.
\end{lemma}
\begin{proof}
We first claim that, if $M$ is a semigroup in Cu and $X\subseteq M$ is a countable subset such that, given $a\in M$ and $a'\ll a$, there is $b\in X$ with $a' \ll b\ll a$, then $M$ is countably based.

Indeed, given $a\in M$ write $a=\sup_n a_n$, where $(a_n)$ is a rapidly increasing sequence in $M$. Then, by our assumption there exists a sequence $(b_n)$ in $X$ such that
\begin{align*}
 a_1\ll b_1\ll a_2\ll b_2\ll a_3\ll \cdots.
\end{align*}
This implies that $(b_n)$ is rapidly increasing and that $\sup_n b_n=\sup_n a_n=a$. 


Let now $A$ be a separable C$^*$-algebra, which we may take to be stable. Let $F$ be a countable dense subset of $A_+$, and consider the set 
\[
X=\{[ (a -1/m)_+]\mid a\in F,m\in\N\}\subseteq \mathrm{Cu}(A)\,.
\]
Given $b\in A_+$ and $x\ll [ b]$, find $m$ such that $x\leq [ (b-1/m)_+]$. For this $m$, there is $a\in F$ such that $\Vert b-a\Vert<1/4m$. From this, we first obtain 
\[
[ (a-1/2m)_+] \ll [ (a-1/4m)_+]\leq [ b]\,.
\]
Observe also that $\Vert (a-1/2m)_+-b\Vert< 1/2m+1/4m=3/4m$, whence
\[
x\leq[ (b-1/m)_+]\ll[ (b-3/4m)_+]\leq [ (a-1/2m)_+]\ll[ b]\,.
\] 
Thus the result follows from the first part of the proof and the fact that $\Cu(A)$ is a semigroup in Cu.
\end{proof}

The following lemma is a modification of Lemma 2 of \cite{Ciuperca-Elliott-Santiago}:

\begin{lemma}\label{lem: unitary}
Let $A$ be a C$^*$-algebra, and let $B$ be a hereditary subalgebra such that $B\subseteq\overline{\mathrm{GL}(B^\sim)}$. 
If $a$ is a positive contraction, and $x_0,x_1\in A$ are such that $x_0x_0^*, x_1x_1^*\in B$, and
\begin{align*}
\|a-x_0^*x_0\|<\epsilon,\quad \|a-x_1^*x_1\|<\epsilon,
\end{align*}
then there exists a unitary $u$ in $B^\sim$ such that
\begin{align*}
\|x_0-ux_1\|<9\epsilon.
\end{align*} 
\end{lemma}
\begin{proof}
By Proposition 1 of \cite{Robert-Santiago}, applied to the elements $a$ and $x_0$, there 
exists $y_0\in A$ such that 
\begin{align*}
(a-\epsilon)_+=y_0^*y_0, \quad \|y_0-x_0\|<4\epsilon, \quad y_0y_0^*\in B.
\end{align*}
(A straightforward computation shows that Proposition 1 of \cite{Robert-Santiago} holds for $C=4$.)
Similarly, there exists $y_1\in A$ such that
\begin{align*}
(a-\epsilon)_+=y_1^*y_1, \quad \|y_1-x_1\|<4\epsilon, \quad y_1y_1^*\in B.
\end{align*}
It follows now from Lemma 2 of \cite{Ciuperca-Elliott-Santiago}, applied to the elements $y_0$ and $y_1$,
that there exists a unitary $u\in B^\sim$ such that
\begin{align*}
\|y_0-uy_1\|<\epsilon.
\end{align*}
(Note that Lemma 2 of \cite{Ciuperca-Elliott-Santiago} still holds if the assumption of $B$ having stable
rank one is replaced by $B\subseteq\overline{\mathrm{GL}(B^\sim)}$.) Therefore,
\begin{align*}
\|x_0-ux_1\|<\|x_0-y_0\|+\|y_0-uy_1\|+\|uy_1-ux_1\|<4\epsilon+\epsilon+4\epsilon=9\epsilon.
\end{align*}
\end{proof}

\subsection{$\CC(X)$-algebras}

Let $X$ be a compact Hausdorff space. Recall that a $\mathrm{C}(X)$-algebra is a C$^*$-algebra $A$ endowed with a unital $\ast$-homomorphism $\theta$ from $\mathrm{C}(X)$ to the center $\mathrm{Z}(\mathrm{M}(A))$ of the multiplier algebra $\mathrm{M}(A)$ of $A$ (see, e.g. \cite{dadarlat}). We shall refer to the map $\theta$ as the \emph{structure map}.

For each closed subset $Y$ of $X$, we define $A(Y)$ to be the quotient of $A$ by the closed two sided ideal $\CC_0(X\setminus Y)A$, which is a $\mathrm{C}(X)$-algebra in the natural way. The quotient map is denoted by $\pi_Y\colon A\to A(Y)$. If, further, $Z\subseteq Y$ is closed then $\pi_Z=\pi^Y_Z\circ\pi_Y$, where $\pi^Y_Z\colon A(Y)\to A(Z)$ denotes the quotient map. In the case that $Y=\{x\}$ the C$^*$-algebra $A(x):=A(\{x\})$ is called the \emph{fibre} of $A$ at $x$, and we write $\pi_x$ for $\pi_{\{x\}}$. The image by $\pi_x$ of an element $a\in A$ is denoted by $a(x)$. It is well known that, for all $a\in A$, the map $x\mapsto \| a(x)\|$ is upper semicontinuous. Moreover, if $a\in A$ one has that $\|a\|=\sup_{x\in X}\|a(x)\|$ and the supremum is attained (see, e.g. \cite[Proposition 2.8]{blanchard}).

The following lemmas will be used in a number of instances, and are quite possibly well known. We include their proofs for completeness.
\begin{lemma}
\label{lem:stablecofx}
Let $A$ be a $\CC(X)$-algebra. Then, $A\otimes\mathcal{K}$ is a $\CC(X)$-algebra and for any closed set $Y$ of $X$ there exists a $*$-isomorphism 
\[
\varphi_Y\colon (A\otimes\mathcal{K})(Y)\to A(Y)\otimes\mathcal K,
\]
such that $\varphi_Y\circ\pi'_Y=\pi_Y\otimes \mathrm{1}_\mathcal{K}$, where $\pi_Y\colon A\to A(Y)$ and $\pi_Y'\colon A\otimes \mathcal{K}\to (A\otimes \mathcal{K})(Y)$ denote the quotient maps. In particular, for any $x\in X$, we have $(A\otimes \mathcal K)(x)\cong A(x)\otimes \mathcal K$ with $(a\otimes k)(x)\mapsto a(x)\otimes k$.
\end{lemma}
\begin{proof}
Abusing the language, define $\theta\colon \CC(X)\to \mathrm{Z}(\mathrm{M}(A\otimes\mathcal K))$ on elementary tensors by $\theta (f)(a\otimes k)=(\theta(f)a)\otimes k$, whenever $a\in A$ and $k\in\mathcal K$, and where $\theta$ is the structure map for $A$. This endows $A\otimes\mathcal K$ with a structure map.

Now, if $Y$ is closed in $X$, we have an exact sequence 
\[
0\to \CC_0(X\!\setminus\!Y)A\to A\to A/\CC_0(X\!\setminus\!Y)A\to 0\,.
\]
Tensoring by the compacts we obtain another exact sequence 
\[
0\to (\CC_0(X\!\setminus\!Y)A)\otimes \mathcal K\to A\otimes\mathcal K\to (A/\CC_0(X\!\setminus\!Y)A)\otimes\mathcal K\to 0\,.
\]
Since $(\CC_0(X\!\setminus\!Y)A)\otimes \mathcal K=\CC_0(X\!\setminus\!Y)(A\otimes \mathcal K)$, it follows that 
\[
(A\otimes \mathcal{K})(Y)=(A\otimes\mathcal K)/\CC_0(X\!\setminus\!Y)(A\otimes \mathcal K)\cong A(Y)\otimes\mathcal K\,.
\]
\end{proof}
\begin{remark}
\label{rem:stable}
{\rm Let $A$ be a $\CC(X)$-algebra. Then, Lemma \ref{lem:stablecofx} implies that the natural map $\pi_x\colon A\to A(x)$ induces, at the level of the Cuntz semigroup, a map $\Cu(A)\to \Cu(A(x))$ that can be viewed as $[a]\mapsto [a(x)]$. In turn, these maps define a map
\[
\alpha\colon \Cu(A)\to \prod_{x\in X}\Cu(A(x))\,.
\]
Similarly, if $Y$ is closed in $X$ the map $\pi_Y$ induces a map $\Cu(\pi_Y)\colon\Cu(A)\to\Cu(A(Y))$ such that $\Cu(\pi_Y)[a]=[a|_Y]$.
Thus, by the previous lemma, when computing the Cuntz semigroup of $A$ we may assume that $A$, $A(x)$ and $A(Y)$ are stable.}
\end{remark}
\begin{lemma}\label{lem:AB}
 Let $A$ be a $\CC(X)$-algebra and let $B:=A+\CC(X)\cdot 1_{\M(A)}$. Then, 
 \begin{enumerate}[{\rm (i)}]
  \item $B$ is a $\CC(X)$-algebra that contains $A$ as a closed two-sided ideal. In particular, $A$ is $\CC(X)$-subalgebra of $B$. 
  
  \item The restriction of $\pi_x\colon B\to B(x)$ to $A$ induces an isomorphism $A(x)\cong \pi_x(A)$ for all $x\in X$. 
 \end{enumerate}
\end{lemma}
\begin{proof}
 (i). Since the quotient of $B$ by the C*-algebra $A$ is a C*-algebra, $B$ itself is a C*-algebra. It is clear that $B$ is a $\CC(X)$-algebra and that $A$ is a closed two-sided ideal of $B$.
 
 The second part of the lemma follows from part (v) of \cite[Lemma 2.1]{dadarlat}.
\end{proof}

\section{The Cuntz Semigroup of $\mathrm{C}([0,1], A)$}

\begin{theorem}\label{th: C[0,1]-algebras}
Let $A$ be a $\mathrm{C}[0,1]$-algebra such that for $t$ in a dense subset of $[0,1]$ the fibre C$^*$-algebra $A(t)$ is separable, has stable rank one, and $\mathrm{K}_1(I)=0$ for any ideal $I$ of $A(t)$. Then, the map
\[
\alpha\colon\Cu(A)\to \prod_{t\in [0,1]}\Cu(A(t)),
\]
given by $\alpha[a]=([a(t)])_{t\in [0,1]}$ is an order embedding.
\end{theorem}

\begin{proof}
By Remark \ref{rem:stable} and since our assumptions on $A$ and its fibres are stable, we may assume that $A$ and its fibres are stable at the outset.

Let $0<\epsilon<1$ be fixed, and let us suppose that $a,b\in A$ are positive contractions such that $a(t)\precsim b(t)$, for all $t\in [0,1]$. We need to show that $a\precsim b$.

Let $B:=A+\CC[0,1]\cdot 1_{\M(A)}$. Then $B$ is a C*-algebra that contains $A$ as a closed two-sided ideal by (i) of Lemma \ref{lem:AB}. In addition, by (ii) of Lemma \ref{lem:AB} we have that $a(t)\precsim b(t)$ in $B(t)$, for all $t\in [0,1]$. By the definition of the Cuntz order and since $B(t)$ is a quotient of $B$ for each $t\in [0,1]$ there exists $d\in B$ such that
\begin{align*}
\|a(t)-d(t)^*b(t)d(t)\|<\epsilon.
\end{align*} 
By the upper semicontinuity of the norm the inequality above also holds in a neighbourhood of $t$. Hence, since $[0,1]$ is a compact set, there exist a finite covering of $[0,1]$ consisting of open intervals $U_i$, $i=1,2,\cdots,n$, and elements $(d_i)_{i=1}^n\subset B$ such that
\begin{align*}
\|a(t)-d_i(t)^*b(t)d_i(t)\|<\epsilon.
\end{align*} 
for all $t\in U_i$ and all $1\le i\le n$. Moreover, we may choose the open intervals $(U_i)_{i=1}^n$ such that $t<t'$ if $t\in U_i$ and $t'\in U_{i+2}$, for $i=1,2,\cdots, n-2$.

For each $1\le i\le n$, set $b^{\frac{1}{2}}d_i=x_i$. Then, $x_ix_i^*\in \Her(b)$, and
\begin{align*}
\|a(t)-x_i(t)^*x_i(t)\|<\epsilon,
\end{align*}
for all $t\in U_i$. By assumption, there exists $t_i\in U_i\cap U_{i+1}$ such that the stable rank of $\pi_{t_i}(A)\cong A(t_i)$ is one, where $\pi_{t_i}\colon B\to B(t_i)$ denotes the quotient map. Therefore, since $\Her(b(t_i))=\overline{b(t_i)B(t_i)b(t_i)}$ is also a hereditary subalgebra of $\pi_{t_i}(A)$, the stable rank of $\Her(b(t_i))$ is one.

We now have
\begin{align*}
&x_i(t_i)x_i(t_i)^*, x_{i+1}(t_i)x_{i+1}(t_i)^*\in \Her(b(t_i)),\\
&\|a(t_i)-x_i(t_i)^*x_i(t_i)\|<\epsilon,\quad \|a(t_i)-x_{i+1}(t_i)^*x_{i+1}(t_i)\|<\epsilon.
\end{align*}
Hence, by Lemma \ref{lem: unitary} there exists a unitary $u_i$ in the unitization of $\Her(b(t_i))$ such that
\begin{align*}
\|x_i(t_i)-u_ix_{i+1}(t_i)\|<9\epsilon.
\end{align*}
Note that since $\pi_{t_i}(A)\cong A(t_i)$ is stable $1_{\M(A)}(t_i)\notin \pi_{t_i}(A)$, whence $1_{\M(A)}(t_i)\notin\Her(b(t_i))$. This implies that the unitization of $\Her(b(t_i))$ is isomorphic to the C*-algebra $\Her(b(t_i))+\C\cdot 1_{\M(A)}(t_i)$. Therefore, we may assume that the unitary $u_i$ belongs to this algebra.

Using \cite[Theorem 2.8]{brown} and that $A(t_i)$ is separable we conclude that $\Her(b(t_i))$ is stably isomorphic to an ideal of $\pi_{t_i}(A)$, which is in turn isomorphic to an ideal of $A(t_i)$. Hence, it follows from our assumptions that $\mathrm{K}_1(\Her(b(t_i)))=0$. Since $\mathrm{sr}(\Her(b(t_i)))=1$, we know from \cite[Theorem 2.10]{Rieffel} that $\mathrm{U}(\Her(b(t_i)))$ is connected. Therefore, $u_i$ can be connected to $1_{\M(A)}(t_i)$ in $\Her(b(t_i))+\C\cdot 1_{\M(A)}(t_i)$. Since this C*-algebra is the image by $\pi_{t_i}$ of the C*-algebra $\Her(b)+\C\cdot 1_{\M(A)}$, and unitaries in the connected component of the identity lift (see, e.g. \cite[Corollary 4.3.3]{wo}), there exists a unitary $v_i$ in $\Her(b)+\C\cdot 1_{\M(A)}$ such that $v_i(t_i)=u_i$.
 
Let $(y_i)_{i=1}^n$ be the elements defined by $y_1=x_1$, and 
\begin{align*}
y_i=v_1v_2\cdots v_{i-1}x_i,
\end{align*}
for $i=2,\cdots,n$. 
Since $x_i\in b^{1/2}A$, it follows that $y_i\in A$ for all $i$. Also, $y_iy_i^*\in \Her(b)$, and
\begin{align}\label{ayy}
\|a(t)-y_i(t)^*y_i(t)\|<\epsilon,
\end{align}
for all $t\in U_i$. Moreover, 
\begin{align*}
\| (y_i-y_{i+1})(t_i)\|=\| v_1\cdots v_{i-1}(x_i-v_ix_{i+1})(t_i)\|=\| (x_i-v_ix_{i+1})(t_i)\|=\| x_i(t_i)-u_ix_{i+1}(t_i)\|\,.
\end{align*}
Thus,
\begin{align*}
\|y_i(t_i)-y_{i+1}(t_i)\|<9\epsilon.
\end{align*}
Since the norm is upper semicontinuous there exists $\delta>0$ such that 
\begin{align*}
\|y_i(t)-y_{i+1}(t)\|<9\epsilon.
\end{align*}
for all $t\in (t_i-\delta, t_i+\delta)$ and for all $1\le i\le n-1$.
Let us consider the open intervals $(V_i)_{i=1}^n$ defined by
\begin{align*}
V_i=
\begin{cases}
[0,t_1+\delta)&\text{if }i=1;\\
(t_{i-1}-\delta, t_i+\delta)&\text{if }2\le i\le n-1;\\
(t_{n-1}-\delta,1]&\text{if }i=n.
\end{cases}\
\end{align*}
Then, $\bigcup_{i=1}^nV_i=[0,1]$, and $V_i\subseteq U_i$ for all $1\le i\le n$.
Let $(\lambda_i)_{i=1}^n$ be a partition of unity associated to the open covering $(V_i)_{i=1}^n$. Let us consider the element
\begin{align*}
y=\sum_{i=1}^n\lambda_iy_i.
\end{align*}
Then, $y\in A$, $yy^*\in \Her(b)$, and
\begin{align}\label{y}
y(t)=
\begin{cases}
y_1(t)&\text{if } 0\le t\le t_1-\delta;\\
y_i(t)&\text{if } t_{i-1}+\delta\le t\le t_i-\delta,\text{ and }2\le i\le n-1;\\
y_n(t)&\text{if } t_{n-1}+\delta\le t\le 1;\\
\lambda_i(t)y_i(t)+\lambda_{i+1}(t)y_{i+1}(t)&\text{if }t_{i}-\delta<t<t_{i}+\delta, \text{ and }1\le i\le n-1.
\end{cases}
\end{align}
Let us show that $\|a-y^*y\|<28\epsilon$. By \eqref{ayy} and \eqref{y} it is enough to show that $\|a(t)-y(t)^*y(t)\|<28\epsilon$ for $t\in (t_{i}-\delta, t_{i}+\delta)$. We have
\begin{align*}
\|a(t)-y(t)^*y(t)\|&=\|a(t)-(\lambda_i(t)y_i(t)+\lambda_{i+1}(t)y_{i+1}(t))^*(\lambda_i(t)y_i(t)+\lambda_{i+1}(t)y_{i+1}(t))\|\\
&\le\|a(t)-y_i(t)^*y_i(t)\|+\\
&+\|y_i(t)^*y_i(t)-(\lambda_i(t)y_i(t)+\lambda_{i+1}(t)y_{i+1}(t))^*(\lambda_i(t)y_i(t)+\lambda_{i+1}(t)y_{i+1}(t))\|\\
&<\epsilon+27\epsilon\\
&=28\epsilon.
\end{align*}
We have found an element $y\in A$ such that $\|a-y^*y\|<28\epsilon$, and $yy^*\in \Her(b)$. This implies by Lemma 2.2 of \cite{Kirchberg-Rordam} that $(a-28\epsilon)_+\precsim b$ in $A$. Therefore, 
\[
[a]=\sup_{\epsilon>0}[(a-28\epsilon)_+]\le [b].
\]
This concludes the proof of the theorem.
\end{proof}

In the particular case of $\CC(X,A)$, for a given $C^*$-algebra $A$, all fibres are naturally isomorphic to $A$, and hence 
the image of the map $\alpha$ in the theorem above can be viewed as functions from $X$ to $\Cu(A)$, that are lower semicontinuous in
 a certain topology.

\begin{proposition}{\rm (\cite[Proposition 3.1]{Tikuisis})}
 Let $A$ be a C$^*$-algebra, $X$ a compact Hausdorff space and $f\in \CC(X,A)$. Then, for any
$a\in A$, the set $\{x\in X \mid [ b] \ll [ f(x)] \}$ is open.  
\end{proposition}

Given a separable C$^*$-algebra $A$, the sets $\{[ a] \in \Cu(A)\mid [ a]\ll [ b]\}$ define a basis 
of a topology in $\Cu(A)$ named the \emph{Scott Topology} (see, e.g. \cite{scottetal}).

\begin{definition}
Let $X$ be a topological space, $M$  a semigroup in Cu, and $f\colon X\to M$ a function. We say that $f$ is \emph{lower semicontinuous} if, 
for all $a\in M$, the set $f^{-1}(a^{\ll}):=\{t\in X \mid a\ll f(t)\}$ is open in $X$. We shall denote the set of all lower semicontinuous functions from $X$ to $M$ by $\Lsc(X,M)$.
\end{definition}

In section \ref{lsc} we prove that in the general case of a second countable finite dimensional topological space $X$  and a countably based semigroup $M$ in Cu,
$\Lsc(X,M)$, equipped with the pointwise order and addition, is a semigroup in the category Cu. 
The key step in the argument is to show that any function $f\in\Lsc(X,M)$ is a supremum of a rapidly increasing sequence of functions that have a special form. We describe these functions in the particular case of the interval $X=[0,1]$,
and refer the reader to section \ref{lsc} for the general case. 

\begin{definition}\label{def:pw01}
Let $A$ be a C$^*$-algebra. Given the following data
\begin{enumerate}[{\rm (i)}]
\item A partition $0=t_0<t_1<\dots<t_{n-1}<t_n=1$ of $[0,1]$ with $n=2r+1$ for some $r\geq 1$,
\item Elements $x_0,\dots,x_{n-1}$ in $M$, with $x_{2i},x_{2i+2}\leq x_{2i+1}$ for $0\leq i\leq r-1$, 
\end{enumerate}
a \emph{piecewise characteristic function}  is a map $g\colon [0,1]\to \Cu(A)$ such that 
\[
g(s)=\begin{cases} x_{2i} \text{ if }  s\in [t_{2i},t_{2i+1}] \\   x_{2i+1} \text{ if }  s\in (t_{2i+1},t_{2i+2})\end{cases}
\]
If moreover $g\ll f$ for some $f\in \Lsc([0,1],\Cu(A))$, we then say that $g$ is a \emph{piecewise characteristic function} for $f$. We denote the set of all such functions by $\chi(f)$.
\end{definition}
It is easily verified that a piecewise characteristic function as above is lower semicontinuous.

\begin{lemma}\label{prop:inv}
 Let $A$ be a separable  stable C$^*$-algebra, $f\in \Lsc([0,1],\Cu(A))$, and $f_1\ll f_2$ be piecewise characteristic functions for $f$. Then, there exists 
a continuous function $g_2\in \CC([0,1],A)$ such that $f_1\ll \alpha([g_2])\leq f_2$ and $\alpha([g_2])\in\chi(f)$.
\end{lemma}

\begin{proof}
Suppose that $f_2$ is described as in Definition \ref{def:pw01} with $x_i=[ a_i]$ for some $a_i\in A$. 
Let $f_{2,\epsilon}$ be the function with the same form as $f_2$ but with $[a_{2i}]$ replaced by $[(a_{2i}-\epsilon)_+]$ for $0\le i\le r$. Note that $f_{2,\epsilon}\in \chi(f)$ and that $f_2=\sup f_{2,\epsilon}$ in $\Lsc([0,1],\Cu(A))$. Hence, since $f_1\ll f_2$, there exists $\epsilon>0$ such that 
$f_1\ll f_{2,\epsilon} \leq f_2$.

Since $[a_{2i}] \leq [a_{2i-1}],\, [a_{2i+1}]$, by  condition (iv) in Proposition \ref{lem:extror} 
there exist unitaries $u_i$ and $v_i$ in $A^\sim$ such that 
\begin{align}\label{herher}
 u_i(a_{2i}-\epsilon)_+u_i^*\in\Her (a_{2i-1}),\quad   v_i(a_{2i}-\epsilon)_+v_i^*\in\Her (a_{2i+1}).
\end{align}

Since $A$ is stable, the unitary group of the multiplier algebra $\M(A)$ is connected in the norm topology (see, e.g. \cite[Corollary 16.7]{wo}). Therefore, for each $i=0,1,\ldots,r$ there exists a continuous path $w_i\colon [0,1]\to \mathrm{U}(\M(A))$ such that $w_i(t)=u_i$ if $t\in [0,t_{2i}]$, and $w_i(t)=v_i$ for $t\in [t_{2i+1},1]$. 


Let $(\lambda_i)_{i=0}^r$ be sequence of continous positive real-value functions on $[0,1]$ that are supported in the open sets 
\[
[0,t_2), (t_{2i-1},t_{2i+2})_{i=1}^ {r-1},(t_{2r-1},1]\,,
\]
respectively (i.e., they are non-zero in each point of the corresponding interval and zero elsewhere). Let us define $g_2\in \CC([0,1],A)$ by 
\[ 
g_2(t)= \sum_{i=0}^{r} (\lambda_i(t)w_i(t)(a_{2i}-\epsilon)_+w_i^*(t)+ \lambda_{i}(t)\lambda_{i+1}(t) a_{2i+1}).
\]
(In the equation above we are taking $\lambda_{r+1}=0$ and $a_{r+1}=0$.)

If $t\in [t_{2j},t_{2j+1}]$ with $0\le j\le r$, then $g_2(t)=\lambda_i(t)w_j(t)(a_{2j}-\epsilon)_+w_j(t)\sim (a_{2j}-\epsilon)_+$. Hence, $\alpha([g_2])(t)=[(a_{2j}-\epsilon)_+]$. 

If $t\in (t_{2j+1},t_{2j+2})$ with $0\le j\le r-1$, then 
\begin{align*} 
g_2(t)  & =  \lambda_j(t)w_j(t)(a_{2j}-\epsilon)_+w_j(t)^*+\lambda_{j+1}(t)w_{j+1}(t)(a_{2j+2}-\epsilon)_+w_{j+1}(t)^*+\lambda_j(t)\lambda_{j+1}(t) a_{2j+1} \\ 
        & = \lambda_j(t)v_j(a_{2j}-\epsilon)_+v_j^*+\lambda_{j+1}(t)u_{j+1}(a_{2j+2}-\epsilon)_+u_{j+1}^*+\lambda_{j}(t)\lambda_{j+1}(t) a_{2j+1}.
\end{align*}
By \eqref{herher} the element $g_2(t)$ belongs to $\Her(a_{2j+1})$, whence $g_2(t)\precsim a_{2j+1}$. Also, we have 
\[
g_2(t)\geq \lambda_{j}(t)\lambda_{j+1}(t) a_{2j+1}\sim a_{2j+1}.
\] 
Therefore, $\alpha([g_2])(t)=[a_{2j+1}]$. 

It follows that $\alpha([g_2])=f_{2,\epsilon}$, which proves the result.
\end{proof}

\begin{theorem}
\label{thm:surj}
Let $A$ be a separable C$^*$-algebra. If the natural map 
\[
\alpha\colon \mathrm{Cu}(\CC([0,1],A))\to \Lsc([0,1],\mathrm{Cu}(A))
\]
is an order embedding, then it is an isomorphism in the category Cu.
\end{theorem}
\begin{proof}
Without loss of generality we may assume that $A$ is stable. In addition,
we only need to prove that $\alpha$ is surjective since by our assumptions this will imply that it is an order-isomorphism, whence an isomorphism in the category Cu. 

Let
$f\in\Lsc([0,1],\mathrm{Cu}(A))$. We know from Proposition \ref{lem:supseq} combined
with Lemma \ref{lem:countbase}
that there is a rapidly increasing sequence of functions $(f_n)$ in $\chi(f)$
such that $f=\sup f_n$. By Lemma \ref{prop:inv},
we may suppose that there exists $g_n\in \CC([0,1],A)_+$ with $\alpha([
g_n])=f_n$. As $\alpha$ is an order-embedding by assumption, 
the sequence $([ g_n])$ is increasing. Let $[g]=\sup[
g_n]$. Then 
\[
\alpha([g])=\sup_n\alpha([ g_n])=\sup_n f_n=f\,,
\]
as desired.
\end{proof}

%
%

From Theorems \ref{th: C[0,1]-algebras} and \ref{thm:surj}, we immediately obtain the corollary below. Much more is true, as will be shown in the next section.

\begin{corollary}\label{cor:intervalo}
Let $A$ be a separable C$^\ast$-algebra with stable rank one such that $\mathrm{K}_1(I)=0$ for every closed two-sided ideal $I$ of $A$.  Then, the map
\[
\alpha\colon\Cu(\mathrm{C}([0,1],A))\to \mathrm{Lsc}([0,1], \Cu(A)),
\]
given by $\alpha([a])(t)=[a(t)]$ is an isomorphism in the category $\Cu$.
\end{corollary}

%
%

%
%
%


\section{Pullbacks}

In this section we extend the previous results to spaces of dimension at most one. En route to our result, we analyse the behavior of the functor $\Cu$ under the formation of certain pullbacks, which we describe below.

Let $A$, $B$, and $C$ be C*-algebras. Let $\pi\colon A\to C$ and $\phi\colon B\to C$ be $\ast$-homomorphisms. We can form the pullback
\begin{align*}
B\oplus_{C}A=\{(b,a)\in B\oplus A\mid \phi(b)=\pi(a)\}.
\end{align*}
By applying the Cuntz functor $\Cu(\cdot)$ to the $\ast$-homomorphisms $\pi$ and $\phi$ we obtain Cuntz semigroup morphisms (in the category $\Cu$)
\[ 
\Cu(\pi)\colon \Cu(A)\to \Cu(C)\,, \text{ and }\Cu(\phi)\colon \Cu(B)\to \Cu(C).
\]
Let us consider the pullback (in the category of ordered semigroups)
\begin{align*}
\Cu(B)\oplus_{\Cu(C)}\Cu(A)=\{([b],[a])\in \Cu(B)\oplus \Cu(A)\mid \Cu(\phi)[b]=\Cu(\pi)[a]\}.
\end{align*}
Then, we have a natural order-preserving map 
\begin{align}
\label{pullbackbeta}
\beta\colon \Cu(B\oplus_{C}A)\to \Cu(B)\oplus_{\Cu(C)}\Cu(A),
\end{align}
defined by $\beta([(b,a)])=([b],[a])$. Observe that since $\Cu(\pi)$ and $\Cu(\phi)$ are maps in $\Cu$, the pullback semigroup $\Cu(B)\oplus_{\Cu(C)}\Cu(A)$ is closed under suprema of increasing sequences. Note also that the map $\beta$ preserves suprema.


\begin{theorem}\label{thm:injectivepullback1}
 Let $A$, $B$, and $C$ be C*-algebras such that $C$ is separable, has stable rank one, and $\mathrm{K}_1(I)=0$ for every closed two-sided ideal $I$ of $C$. Let $\phi\colon B\to C$ and $\pi\colon A\to C$ be $\ast$-homomorphisms such that $\pi$ is surjective. Then, the map
 \begin{align*}
  \beta\colon \Cu(B\oplus_{C}A)\to \Cu(B)\oplus_{\Cu(C)}\Cu(A),
 \end{align*}
given by $\beta[(b,a)]=([b],[a])$ is an order embedding.
\end{theorem}
\begin{proof}
By \cite[Theorem 3.9]{pedersenpullbacks} applied to  $Y=\K$ we may assume that $A$, $B$, and $C$ are stable. Let $(b_1, a_1)$ and $(b_2, a_2)$ be positive contractions of $B\oplus_{C}A$ such that $a_1\precsim a_2$ and $b_1\precsim b_2$. Let $0<\epsilon<1$. Then, by the definition of the Cuntz relation there are $x\in A$ and $y\in B$ such that
\begin{align*}
& \|a_1-x^*x\|<\epsilon, \quad xx^*\in \Her(a_2),\\
& \|b_1-y^*y\|<\epsilon, \quad yy^*\in \Her(b_2).
\end{align*}
Since $\pi(a_1)=\phi(b_1)$ and $\pi(a_2)=\phi(b_2)$, the equations above imply that
\begin{align}
 &\|\pi(a_1)-\pi(x)^*\pi(x)\|<\epsilon,\quad \|\pi(a_1)-\phi(y)^*\phi(y)\|<\epsilon,\nonumber\\ 
 &\pi(x)\pi(x)^*,\phi(y)\phi(y)^*\in \Her(\pi(a_2)).\label{phiy}
\end{align}
By Lemma \ref{lem: unitary} there is a unitary $u\in \Her(\pi(a_2))^\sim$ such that
\begin{align*}
 \|u\pi(x)-\phi(y)\|<9\epsilon.
\end{align*}
Using \cite[Theorem 2.8]{brown} and that $C$ is separable it follows that $\Her(a_2)$ is stable isomorphic to an ideal of $C$. Hence, by our assumptions $\mathrm{K}_1(\Her(\pi(a_2)))=0$. Since $\mathrm{sr}(A)=1$, we have by \cite[Theorem 2.10]{Rieffel} that $\mathrm{U}(\Her(\pi(a_2)))=\mathrm{U}_0(\Her(\pi(a_2)))$. Therefore, $u$ is in the connected component of the identity. By the surjectivity of the map $\pi$ there exists a unitary $v\in \Her(a_2)^\sim$ such that $\widetilde{\pi}(v)=u$, where $\widetilde{\pi}\colon A^\sim\to C^\sim$ is the extension of $\pi$ to the unitization of the algebras $A$ and $C$. In addition, there exists $y'\in \Her(a_2)$ such that $\pi(y')=\phi(y)$. Hence, we have
\begin{align*}
 \|\pi(vx-y')\|=\|u\pi(x)-\phi(y)\|<9\epsilon.
\end{align*}
Since $vx-y'\in \overline{a_2A}$ there exists $z'\in \overline{a_2A}\cap \mathrm{Ker}(\pi)$ such that
\begin{align*}
 \|vx-y'-z'\|< 9\epsilon.
\end{align*}
Set $y'+z'=z$. Then,
\[
 \pi(z)=\phi(y),\quad zz^*\in \Her(a_2),\quad\|vx-z\|<9\epsilon.
\]
Also,
\begin{align*}
 \|a_1-z^*z\|&\le\|a_1-x^*x\|+\|x^*x-z^*z\|\\
 &< \epsilon+\|(vx)^*(vx)-z^*z\|\\
 &\le \epsilon+\|(vx)^*(vx-z)\|+\|(vx-z)^*z\|\\
 &\le \epsilon+\|vx\|\|vx-z\|+\|z\|\|vx-z\|\\
 &\le \epsilon+2\|vx-z\|+11\|vx-z\|\\
 &<118\epsilon.
\end{align*}
Since $\pi(z)=\phi(y)$ the element $(y,z)$ belongs to $B\oplus_{C}A$, and by the previous computation
\begin{align*}
 \|(b_1, a_1)-(y,z)^*(y,z)\|<118\epsilon.
\end{align*}
In addition, since $yy^*\in \Her(b_2)$ and $zz^*\in \Her(a_2)$ we have
\begin{align*}
(y,z)(y,z)^*=\lim_{n\to \infty}(b_2, a_2)^\frac{1}{n}(y,z)(y,z)^*(b_2, a_2)^\frac{1}{n}\in \Her((b_2,a_2)).
\end{align*}

By \cite[Lemma 2.2]{Kirchberg-Rordam} we have
\begin{align*}
 ((b_1, a_1)-118\epsilon)_+\precsim (y,z)^*(y,z)\sim (y,z)(y,z)^*\precsim (b_2,a_2).
\end{align*}
Therefore,
\begin{align*}
 [(b_1,a_1)]=\sup_{\epsilon>0}[((b_1, a_1)-118\epsilon)_+]\le [(b_2,a_2)].
\end{align*}
\end{proof}

\begin{theorem}
\label{thm:injectivepullback}
Let $X$ be a one-dimensional compact Hausdorff space and let $Y$ be a closed subset of $X$. Let $A$ be a $\CC(X)$-algebra and let $\pi_Y\colon A\to A(Y)$ be the quotient map. Let $B$ be a C$^*$-algebra and let $\phi\colon B\to A(Y)$ be a $\ast$-homomorphism.
Suppose that, for every $x\in X$, the C$^*$-algebra $A(x)$ is separable, has stable rank one, and $\mathrm{K}_1(I)=0$ for every two-sided ideal $I$ of $A(x)$. Then, the map 
\[
 \beta\colon \Cu(B\oplus_{A(Y)}A)\to \Cu(B)\oplus_{\Cu(A(Y))}\Cu(A),
\]
given by $\beta[(b,a)]=([b],[a])$ is an order embedding.
\end{theorem}
\begin{proof}
By Remark \ref{rem:stable} and \cite[Theorem 3.9]{pedersenpullbacks} we may assume that $A$, $A(Y)$ and $B$ are stable.

Let $(b_1,a_1)$ and $(b_2,a_2)$ be positive elements of $B\oplus_{A(Y)}A$ such that $a_1\precsim a_2$ and $b_1\precsim b_2$. Let $\epsilon>0$. By the definition of the Cuntz order, there exist $d\in B$ and $c\in A$ such that 
\begin{align}\label{acac}
\|a_1-c^*a_2c\|<\epsilon, \quad \|b_1-d^*b_2d\|<\epsilon.
\end{align}
Since $\phi(b_1)=\pi_Y(a_1)$ and $\phi(b_2)=\pi_Y(a_2)$, the second inequality in the equation above implies that
\begin{align*}
\|\pi_Y(a_1)-\phi(d)^*\pi_Y(a_2)\phi(d)\|<\epsilon.
\end{align*}
Choose an element $h\in A$ such that $\pi_Y(h)=\phi(d)$. Then, we have
\begin{align}\label{ahah}
\|a_1(x)-h(x)^*a_2(x)h(x)\|<\epsilon,
\end{align}
for all $x\in Y$. By upper semicontinuity, there exists an open neighbourhood $U$ of $Y$ such that the inequality above holds for all $x\in U$. Since $Y\subseteq X$ is compact and $X$ is normal there exists an open subset $V$ such that $Y\subseteq V\subseteq \overline{V}\subseteq U$. Moreover, 
by Theorem \cite[4.2.2]{engelkin} we may assume that $V$ has an empty or zero-dimensional boundary.

%
%


Let $D:=A+\CC(X)\cdot 1_{\M(A)}$. Then, $D$ is a $\CC(X)$-algebra that contains $A$ as a closed two-sided ideal by the first part of Lemma \ref{lem:AB}. Consider the elements $y_1=a_2^{\frac{1}{2}}c$ and $y_2=a_2^{\frac{1}{2}}h$. Then $y_1y_1^*, y_2y_2^*\in \Her(a_2)$. Moreover, rewriting the first inequality in \eqref{acac} and the inequality \eqref{ahah}, we have 
\begin{align*}
\|a_1-y_1^*y_1\|<\epsilon, \quad \|a_1(x)-y_2(x)^*y_2(x)\|<\epsilon,
\end{align*}
for all $x\in \overline{V}$, where the second inequality holds in $D(x)$ (here we are using that $A(x)\cong \pi_x(A)$ by the second part of Lemma \ref{lem:AB}, where $\pi_x\colon D\to D(x)$ denotes the quotient map). In particular, if $x\in \mathrm{bd}(V)$ we have
$y_1(x)y_1(x)^*, y_2(x)y_2(x)^*\in \Her(a_2(x))$, and
\begin{align}\label{a1a1}
\|a_1(x)-y_1(x)^*y_1(x)\|<\epsilon, \quad \|a_1(x)-y_2(x)^*y_2(x)\|<\epsilon.
\end{align}

Since by assumption $A(x)$ has stable rank one, the hereditary algebra $\Her(a_2(x))$ has stable rank one. In addition, since $A(x)$ is stable and $A\cong\pi_x(A)$ we have that $1_{\M(A)}(x)\notin\pi_x(A)$, whence $1_{\M(A)}(x)\notin \Her(a_2(x))$. This implies that $\Her(a_2(x))+\C\cdot 1_{\M(A)}\cong\Her(a_2(x))^\sim$. By Lemma \ref{lem: unitary} applied to \eqref{a1a1}, there exists a unitary $u_x\in \Her(a_2(x))+\C\cdot 1_{\M(A)}(x)$ such that
\begin{equation}
\label{uxyu}
\|u_xy_1(x)-y_2(x)\|<9\epsilon.
\end{equation}
Since the $\mathrm{K}_1$-group of every ideal of $A(x)$ is trivial, it follows that $\mathrm{K}_1(\Her(a_2(x)))=0$ by \cite[Theorem 2.8]{brown} and the separability of $A(x)$. Hence, by \cite[Theorem 2.10]{Rieffel} every unitary in $\Her(a_2(x))+\C\cdot 1_{\M(A)}(x)$ is connected to the identity, and in particular $u_x$. Since $\Her(a_2(x))+\C\cdot 1_{\M(A)}(x)$ is the image of $\Her(a_2)+\C\cdot 1_{\M(A)}$ by $\pi_x$ and unitaries in the connected component of the identity lift, there exists a unitary $v^x\in \mathrm{U}_0(\Her(a_2)+\C\cdot 1_{\M(A)})$ such that $v^x(x)=u_x$.

Suppose first that  $\bd(V)\neq\varnothing$. Note that, by \eqref{uxyu}, 
\[
\|(v^xy_1-y_2)(x)\|=\|u_xy_1(x)-y_2(x)\|<9\epsilon\,,
\] 
for all $x\in X$. Hence by the upper semicontinuity of the norm and since $\mathrm{bd}(V)$ is zero dimensional and compact, there are points $x_1,\ldots,x_n\in \mathrm{bd}(V)$ and an open cover of $\mathrm{bd}(V)$ consisting of pairwise disjoint neighbourhoods $(V_i)_{i=1}^n$ with $x_i\in V_i$ such that
\begin{align}\label{vyy}
\|(v^{x_i}y_1-y_2)(x)\|<9\epsilon,
\end{align}
for all $i$ and all $x\in V_i$. Since the sets $(V_i)_{i=1}^n$ are open, closed, pairwise disjoint, and form a cover of $\mathrm{bd}(V)$, the C*-algebra $D(\mathrm{bd}(V))$ can be identified with the C*-algebra $\bigoplus_{i=1}^nD(V_i)$. Let us consider the element $v=\bigoplus_{i=1}^n\pi_{V_i}(v^{x_i})\in D(\mathrm{bd}(V))$ (here we are using the previous identification). Then, 
$v$ is a unitary in $\pi_{\mathrm{bd}(V)}(\Her(a_2)+\C\cdot 1_{\M(A)})$ that is connected to the identity $\pi_{\mathrm{bd}(V)}(1_{\M(A)})$. Hence, there is a unitary $u\in \Her(a_2)+\C\cdot 1_{\M(A)}$ such that $\pi_{\mathrm{bd}(V)}(u)=v$. By \eqref{vyy} we have
\[
\|u(x)y_1(x)-y_2(x)\|<9\epsilon\,,
\]
for all $x\in\mathrm{bd}(V)$.

Set $uy_1=y'_1$. Since $y_1\in \overline{a_2A}$, It follows that $y_1'\in \overline{a_2A}$. Further, $y_1'y_1'^*\in \Her(a_2)$, and 
\begin{align*}
\|a_1-(y'_1)^*y'_1\|<\epsilon, \quad \|y'_1(x)-y_2(x)\|<9\epsilon.
\end{align*}
for all $x\in \mathrm{bd}(V)$. By upper semicontinuity of the norm there exist an open neighbourhood $W$ of $\mathrm{bd}(V)$ such that 
$W\cap Y=\varnothing$, and 
\begin{align}\label{yy}
\|y'_1(x)-y_2(x)\|<9\epsilon,
\end{align}
for all $x\in W$. Let $f_1,f_2\in \mathrm{C}(X)$ be a partition of unity associated to the covering of $X$ given by the  open sets $V^c\cup W$ and $V\cup W$. 
In case $\bd(V)=\varnothing$ we proceed analogously with $u=1_{\M(A)}$ and $W=\varnothing$, since now $V$, $V^c$ are
both open sets. 

Consider the the element $z=f_1y_1'+f_2y_2$. Then, $z\in A$ and $zz^*\in \Her(a_2)$. Next, using \eqref{yy}, a computation analogous to the one carried out in the proof of Theorem \ref{th: C[0,1]-algebras} shows that
\begin{align*}
\|z^*z-a_1\|<28\epsilon.
\end{align*}
By construction $\pi_Y(z)=\phi(b_2^{\frac{1}{2}}d)$, whence the element $(z,b_2^{\frac{1}{2}}d)\in B\oplus_{A(Y)}A$, and
\begin{align*}
\|(z,b_2^{\frac{1}{2}}d)^*(z,b_2^{\frac{1}{2}}d)-(a_1,b_1)\|<28\epsilon.
\end{align*}
In addition, since $zz^*\in \Her(a_2)$ and $b_2^{\frac{1}{2}}d(b_2^{\frac{1}{2}}d)^*\in\Her(b_2)$, we have
\begin{align*}
 (z,b_2^{\frac{1}{2}}d)(z,b_2^{\frac{1}{2}}d)^*=\lim_{n\to \infty}(a_2,b_2)^\frac{1}{n}(z,b_2^{\frac{1}{2}}d)(z,b_2^{\frac{1}{2}}d)^*(a_2,b_2)^\frac{1}{n}\in \Her((a_2,b_2)).
\end{align*}
Hence, by \cite[Lemma 2.2]{Kirchberg-Rordam}
\begin{align*}
((a_1,b_1)-28\epsilon)_+\precsim (a_2,b_2).
\end{align*}
Therefore,
\begin{align*}
[(a_1,b_1)]=\sup[((a_1,b_1)-28\epsilon)_+]\le[(a_2,b_2)].
\end{align*}
\end{proof}

\begin{theorem}
\label{thm:surjectivepullback}
Let $X$ be a compact Hausdorff space and let $Y$ be a closed subset of $X$. Let $A$ be a $\CC(X)$-algebra, and let $B$ be any C$^*$-algebra. Suppose that the map 
\begin{align*}
\alpha\colon \Cu(A)\to \prod_{x\in X}\Cu(A(x)),
\end{align*}
given by $\alpha([a])(x)=[a(x)]$ is an order embedding. Then
\begin{enumerate}[{\rm (i)}]
\item The map 
\begin{align*}
\beta\colon \Cu(B\oplus_{A(Y)}A)\to \Cu(B)\oplus_{\Cu(A(Y))}\Cu(A),
\end{align*}
defined by $\beta([(b,a)])=([b],[a])$ is surjective. 
\item The pullback semigroup $\Cu(B)\oplus_{\Cu(A(Y))}\Cu(A)$ is in the category $\Cu$.
\end{enumerate}
\end{theorem}

\begin{proof}
(i). By Remark \ref{rem:stable} and \cite[Theorem 3.9]{pedersenpullbacks} we may assume that $A$, $A(Y)$, and $B$ are stable. Let $a\in A$ and $b\in B$
be positive elements such that $\pi_Y(a)\sim \phi(b)$. Choose a positive element $c\in A$ such that $\pi_Y(c)=\phi(b)$. Then we have $\pi_Y(a)\sim \pi_Y(c)$.

Let $\epsilon>0$.
Since $\pi_Y(a)\precsim \pi_Y(c)$, by (iii) of Proposition \ref{lem:extror} there exists $0<\delta<\epsilon$ such that $\pi_Y((a-\epsilon)_+)\precsim \pi_Y((c-\delta)_+)$. Therefore, by the definition of the Cuntz order and since $\pi_Y$ is surjective
there exists $d\in A$ such that 
\[
\|\pi_Y(a-\epsilon)_+-\pi_Y(d)^*\pi_Y(c-\delta)_+\pi_Y(d)\|<\epsilon.
\]
In particular, in the fibre algebras $A(x)$, with $x\in Y$, we have 
\begin{align*}
\|(a-\epsilon)_+(x)-d(x)^*(c-\delta)_+(x)d(x)\|<\epsilon.
\end{align*}
By upper semicontinuity of the norm, there exists an open neighbourhood $U$ of $Y$ such that the inequality above holds for all $x\in U$. Since $X$ is normal there exists an open set $W$ such that $Y\subseteq W\subseteq \overline{W}\subseteq U$. Without loss of generality we may assume that $U=W$ and that 
\[
\|(a-\epsilon)_+(x)-d(x)^*(c-\delta)_+(x)d(x)\|<\epsilon,
\]
holds for all $x\in \overline{U}$. It follows now that
\begin{align*}
\|\pi_{\overline U}((a-\epsilon)_+-d^*(c-\delta)_+d)\|<\epsilon.
\end{align*}
By \cite[Lemma 2.2]{Kirchberg-Rordam} and since $\pi_{\overline{U}}$ is surjective, there exists $f\in A$ such that $\pi_{\overline{U}}((a-2\epsilon)_+)=\pi_{\overline{U}}(f^*(c-\delta)_+f)$. This implies that
\begin{align}
&\pi_{\overline{U}}((a-2\epsilon)_+)\precsim \pi_{\overline{U}}((c-\delta)_+),\label{triangulo}\\
&\pi_{\overline{U}}((a-3\epsilon)_+)=\pi_{\overline{U}}((f^*(c-\delta)_+f-\epsilon)_+)\label{cuadrado}.
\end{align}
Since $f^*(c-\delta)_+f\precsim (c-\delta)_+$, by (iv) of Proposition \ref{lem:extror} there exists a unitary $u\in A^\sim$ such that \begin{align}\label{circulo}
u^*(f^*(c-\delta)_+f-\epsilon)_+u\in \Her((c-\delta)_+).
\end{align}

Let us consider the element $a'=u^*au$. Then, equations \eqref{cuadrado} and \eqref{circulo} imply that
\begin{align*}
\pi_{\overline{U}}((a'-3\epsilon)_+)\in \pi_{\overline{U}}(\Her((c-\delta)_+))=\Her(\pi_{\overline{U}}((c-\delta)_+)).
\end{align*}
Hence, passing to the fibres we have
\begin{align}\label{a'1}
(a'-3\epsilon)_+(x)\in \Her((c-\delta)_+(x)),
\end{align}
for all $x\in U$. In addition, by \eqref{triangulo} we have
\begin{align}\label{a'2}
(a'-2\epsilon)_+(x)\precsim (c-\delta)_+(x),
\end{align}
for all $x\in U$.

Now let us use that $\pi_Y(c)\precsim\pi_Y(a)\sim \pi_{Y}(a')$. Arguing as above, there exists $g\in A$ such that 
\begin{align*}
\|c(x)-g(x)^*a'(x)g(x)\|<\delta'<\delta,
\end{align*}
for all $x$ in a closed neighbourhood $\overline V$ of $Y$. Without loss of generality, we may assume that $U=V$. 
Therefore,
\[
\|\pi_{\overline U}(c-g^*a'g)\|<\delta'<\delta\,.
\]
Hence, by \cite[Lemma 2.2]{Kirchberg-Rordam} we have $\pi_{\overline U}((c-\delta')_+)\precsim \pi_{\overline U} (a')$. This implies by (iii) of Proposition \ref{lem:extror} that $\pi_{\overline U}((c-\delta)_+)\precsim \pi_{\overline U}((a'-3\epsilon')_+)$, for some $\epsilon'<\epsilon$. Now passing to the fibres we have that
\begin{align}\label{a'3}
(c-\delta)_+(x)\precsim (a'-3\epsilon')_+(x),
\end{align}
for all $x\in U$.

Let $f_1, f_2\in \mathrm{C}(X)$ be a partition of unity associated to the open sets $U$ and $X\setminus Y$. Let us consider the element 
\[
z=f_1(c-\delta)_++f_2(a'-3\epsilon)_+. 
\]
Then, 
\begin{eqnarray}\label{zzzz}
\begin{aligned}
&z(x)=(c(x)-\delta)_+&\text{if }x\in Y,\\
&z(x)=(a'(x)-3\epsilon)_+&\text{if }x\in X\setminus U,\\
&z(x)\succsim (c-\delta)_+(x) \text{ or } z(x)\succsim (a'-3\epsilon)_+(x)&\text{if }x\in U,\\
&z(x)\in \Her((c(x)-\delta)_+))&\text{if }x\in U,
\end{aligned}
\end{eqnarray}
where the last equation follows by \eqref{a'1}.

By the choice of $c$ and the first equation in \eqref{zzzz} we have $z(x)=(c-\delta)_+(x)=(\phi(b)-\delta)(x)$, for all $x\in Y$.  Hence, $\pi_Y(z)=\phi((b-\delta)_+)$, and so $((b-\delta)_+,z)$ is an element of the pullback.
We also obtain by the first and third equation of \eqref{zzzz} and by \eqref{a'2} that $(a'-3\epsilon)_+(x)\precsim z(x)$, for all $x\in X$. In addition, by the first and last equation of \eqref{zzzz} and by \eqref{a'3} we obtain that $z(x)\precsim (a'-3\epsilon')_+(x)$, for all $x\in X$. Since by assumption the map $\alpha$ is an order embedding, this yields
\[
(a'-3\epsilon)_+\precsim z\precsim (a'-3\epsilon')_+.
\]

Therefore, we can choose sequences $\delta_n<\epsilon_n$ decreasing to zero and elements $z_n$ in $A_+$ such that $((b-\delta_n)_+, z_n)\in B\oplus_{A(Y)} A$, 
\[
((b-\delta_n)_+, z_n)\precsim ((b-\delta_{n+1})_+, z_{n+1}),
\]
for all $n$, $\sup_n[(b-\delta_n)_+]=[b]$, and $\sup_n[z_n]=[a']=[a]$.  Moreover, $([(b-\delta_n)_+],[z_n])$ is by construction rapidly increasing and $\sup_n([(b-\delta_n)_+],[z_n])=([b],[a])$. 
It also follows that 
\[
\beta(\sup_n([((b-\delta_n)_+,z_n)])=\sup_n\beta([((b-\delta_n)_+,z_n)])=\sup_n([(b-\delta_n)_+],[z_n])=([b],[a])\,,
\]
which proves that $\beta$ is surjective.

(ii). We need to show that $\Cu(B)\oplus_{\Cu(A(Y))}\Cu(A)$ satisfies the axioms of the category Cu, that is to say, (i), (ii), and (iii) of Theorem \ref{th:CEI}.  It was previously shown in the proof of the first part of the theorem that the pullback satisfies (ii). That the pullback satisfies the rest of the axiom follows easily using this fact and that $\Cu(\pi_Y)$ and $\Cu(\phi)$ are morphisms in the category Cu. 
\end{proof}

We now turn our attention to the $\CC(X)$-algebras of the form $\CC(X,A)$. In order to deal with general one-dimensional spaces, we will first analyse the case where the underlying space is a graph. These algebras can be conveniently described in pullback form, as follows (see, e.g. \cite[Section 3.1]{elp} combined with \cite[Theorem 3.8]{pedersenpullbacks}). As a directed graph, write $X=(V, E, r, s)$, where $V=\{v_1,\ldots,v_n\}$ is the set of vertices, $E=\{e_1,\ldots, e_m\}$ is the set of edges, and $r,s\colon E\to V$ are the range and source maps. For $1\leq k\leq m$, denote by $i_k\colon A\to A^m\oplus A^m$ the inclusion in the $k$th component of the first summand. Likewise, we may define $j_k\colon A\to A^m\oplus A^m$ for the second summand. Next, define
\[
\phi\colon \CC(V,A)\to A^m\oplus A^m
\]
by
\[
\phi(g)=\sum\limits_{l=1}^n\big(\sum\limits_{k\in s^{-1}(v_l)} i_k(g(v_l))+\sum\limits_{k\in r^{-1}(v_l)} j_k(g(v_l))\big)\,.
\]
Finally, let $\pi_{\{0,1\}}\colon \CC([0,1], A)\to \CC(\{0,1\},A)$ denote the quotient map.
Then $$\CC(X,A)\cong \CC([0,1], A^m)\oplus_{A^m\oplus A^m}\CC(V,A)$$ (where $A^m\oplus A^m$ is identified with $\CC(\{0,1\}, A^m)$ in the obvious manner).

\begin{theorem}
\label{thm:ese}
Let $X$ be a locally compact Hausdorff space that is second countable and one-dimensional. Let $A$ be a separable C$^*$-algebra with stable rank one such that 
$\mathrm{K}_1(I)=0$ for every closed two-sided ideal $I$ of $A$. Then, the map $\alpha\colon\Cu(\CC_0(X,A))\to \Lsc(X,\Cu(A))$ 
given by $\alpha([a])(x)=[a(x)]$, for all $a\in \CC_0(X,A)$ and $x\in X$,
is an isomorphism in the category Cu.
\end{theorem}

\begin{proof}
By Corollary \ref{cor:intervalo}, the result holds when $X=[0,1]$. Now, let $X$ be a finite graph. By Theorems \ref{thm:surjectivepullback} and \ref{thm:injectivepullback} and the comments previous to this theorem
\begin{align*}
\Cu(\CC(X,A)) & \cong  \Cu(\CC([0,1], A^m))\oplus_{\Cu(A^m\oplus A^m)}\Cu(\CC(V, A))  \\ 
              &\cong  \Lsc([0,1], \Cu(A^m))\oplus_{\Cu(A^m\oplus A^m)}\Lsc(V, \Cu(A))\\
              &\cong \Lsc(X, \Cu(A)).
\end{align*}
Note that the isomorphism between $\Cu(\CC(X,A))$ and $\Lsc(X, \Cu(A))$ obtained above is given by the map $[a]\mapsto (x\in X\mapsto [a(x)])$.

Next, any compact Hausdorff space $X$ that is second countable and one-dimensional can be written 
as a projective limit $X=\varprojlim (X_i,\mu_{i,j})_{i,j\in\N}$, where $X_i$ are finite graphs and $\mu_{i,j}\colon X_j\to X_i$, with $i\le j$, are surjective maps (see \cite[pp. 153]{engelkin}). By \cite[Theorem 2]{CEI} this implies that 
\[
\Cu(\CC(X,A))=\varinjlim (\Cu(\CC(X_i,A)), \Cu(\rho_{i,j}))_{i,j\in \N}, 
\]
where $\rho_{i,j}\colon \CC(X_i)\to \CC(X_j)$, with $i\le j$, is the $\ast$-homomorphism induced by $\mu_{i,j}$, i.e. $\rho(f)=f\circ\mu_{i,j}$. In addition, by (i) of Proposition \ref{prop:limits} we have
\[
\Lsc(X,\Cu(A))=\varinjlim (\Lsc(X_i,\Cu(A)), \Lsc(\mu_{i,j}))_{i,j\in \N}, 
\]
where the maps $\Lsc(\mu_{i,j})\colon \Lsc(X_i, \Cu(A))\to \Lsc(X_j, \Cu(A))$ are given by $\Lsc(f)=f\circ\mu_{i,j}$.

Consider the following diagram:
\begin{align*}
 &\xymatrix@C=40pt{
        \Cu(\CC(X_1,A))\ar[r]^{\Cu(\rho_{1,2})}\ar[d]^\alpha&
        \Cu(\CC(X_2,A))\ar[r]^(0.65) {\Cu(\rho_{2,3})}\ar[d]^\alpha&
        \cdots \ar[r]&
        \Cu(\CC(X,A))\ar[d]^\alpha\\
        \Lsc(X_1,\Cu(A))\ar[r]^{\Lsc(\mu_{1,2})}&
        \Lsc(X_2,\Cu(A))\ar[r]^(0.65){\Lsc(\mu_{2,3})}&
        \cdots\ar[r]&
        \Lsc(X,\Cu(A))
        }
\end{align*}
This diagram is clearly commutative. Hence, since by the argument above the vertical arrows in the finite stages are  isomorphisms the map between the limit semigroups is an isomorphism. This proves the theorem in the case that $X$ is compact.  
 
Let $X$ is an arbitrary locally compact space. Then, applying the first part of the proof to its one-point compactification $\widetilde X=X\cup\{\infty\}$ we conclude that the map $\alpha\colon\Cu(\CC(\widetilde X,A))\to\Lsc(\widetilde X,\Cu(A))$ is an isomorphism in the category Cu. It is easy to check that the image by $\alpha$ of the order-ideal $\Cu(\CC_0(X,A))$ of $\Cu(\CC(\widetilde X,A))$ is $\{f\in\Lsc(\widetilde X,\Cu(A)\mid f(\infty)=0\}$. The latter is in turn order isomorphic to $\Lsc(X,\Cu(A))$ (via restriction). Thus, the result follows.
\end{proof}
\begin{corollary}\label{cor:aquel}
Let $X$ be a compact Hausdorff space that is second countable and one-dimensional. Let $A$ be a separable C$^*$-algebra with stable rank one such that $\mathrm{K}_1(I)=0$ for every closed two-sided ideal $I$ of $A$. Let $B$ be any C$^*$-algebra and suppose $\phi\colon B\to \CC(Y,A)$ is a $*$-homomorphism, where $Y\subseteq X$ is a closed subset of $X$. Then
\[
\Cu(B\oplus_{\CC(Y,A)} \CC(X,A))\cong \Cu(B)\oplus_{\Lsc(Y,\Cu(A))}\Lsc(X,\Cu(A)),
\]
in the category Cu.
\end{corollary}
\begin{proof}
Combine Theorems \ref{thm:injectivepullback}, \ref{thm:surjectivepullback} and \ref{thm:ese}.
\end{proof}

\begin{corollary}\label{cor: pullbackdimtwo}
Let $X$ be a second countable, compact Hausdorff space of dimension at most two such that its second \v Cech cohomology group $\check{\mathrm{H}}^2(X,\Z)$ vanishes. Let $Y$ be a closed subspace of $X$ of dimension zero, let $A$ be a AF-algebra and let $B$ be an arbitrary C*-algebra. If $\pi\colon \mathrm{C}(X,A)\to \mathrm{C}(Y, A)$ is the quotient map and $\phi\colon B\to \mathrm{C}(Y, A)$ is a $\ast$-homomorphism, then
\begin{align*}
 \Cu(B\oplus_{\mathrm{C}(Y,A)}\mathrm{C}(X,A))\cong \Cu(B)\oplus_{\Lsc(Y, \Cu(A))}\Lsc(X,\Cu(A)),
\end{align*}
in the category $Cu$.
\end{corollary}
\begin{proof}
Since $A$ is an AF-algebra, $A$ admits an inductive limit decomposition
\begin{align*}
 A_1\to A_2\to\cdots\to A,
\end{align*}
where the C*-algebras $A_i$, $i=1,2,\dots$, are finite dimensional. By \cite[Theorem 2]{CEI} and by (ii) of Proposition \ref{prop:limits} we have $\Cu(A)=\varinjlim \Cu(A_i)$ and $\Lsc(X,\Cu(A))=\varinjlim \Lsc(X, \Cu(A_i))$.

Consider the following commutative diagram:
\[
\xymatrix
{\Cu(\CC(X,A_1))\ar[d]\ar[r] & \Cu(\CC(X,A_2))\ar[d]\ar[r] & \dots\ar[r]  & \Cu(\CC(X,A))\ar[d]\\
\Lsc(X,\Cu(A_1))\ar[r] & \Lsc(X,\Cu(A_2))\ar[r] & \dots\ar[r]  & \Lsc(X,\Cu(A))}
\]
where the vertical arrows are given by the Cuntz semigroup morphisms induced by the rank function. These maps are isomorphisms by \cite[Theorem 1]{leonel}. Hence, they induce an isomorphism between the limit semigroups, that is, $\Cu(\mathrm{C}(X,A))\cong \Lsc(X,\Cu(A))$. 

Since $Y$ is compact and zero-dimensional, and $A$ is AF, we see that $\CC(Y,A)$ is AF. It thus follows that $\mathrm{K}_1(I)=0$ for any ideal $I$ of $\CC(Y,A)$. The corollary now follows from Theorems
\ref{thm:injectivepullback1} and \ref{thm:surjectivepullback}.
\end{proof}

\begin{corollary}
Let $X$ be a locally compact Hausdorff space that is second countable and one-dimensional. Let $A$ be a separable C$^*$-algebra with stable rank one such that 
$\mathrm{K}_1(I)=0$ for every closed two-sided ideal $I$ of $A$. Then, the semigroup $\Cu(\CC_0(X,A))$ is order-cancellative with respect to $\ll$.
\end{corollary}
\begin{proof}
By Theorem \ref{thm:ese}, $\alpha\colon\Cu (\CC_0(X,A))\to\Lsc(X, \Cu(A))$ is an order-isomorphism. Let $[a], [b], [c]\in \Cu(\CC_0(X,A))$ be such that \[[a]+[b]\ll [a]+[c]\,.\] There exists then $\epsilon>0$ with $[a]+[b]\leq [(a-\epsilon)_+]+[(c-\epsilon)_+]$. Applying $\alpha$ we obtain \[[a(x)]+[b(x)]\leq [(a(x)-\epsilon)_+]+[(c(x)-\epsilon)_+]\] for all $x\in X$.  Using now \cite[Theorem 4.3]{Rordam-Winter}, we conclude that $[b(x)]\leq [(c(x)-\epsilon)_+]$ for all $x\in X$. Since $\alpha$ is an order-isomorphism, we get $[b]\leq [(c-\epsilon)_+]$, so that $[b]\ll [c]$, as desired.
\end{proof}
Recall that an element $a$ in an ordered semigroup is \emph{compact} 
if $a\ll a$. Compact elements in $\Cu(A)$ are strongly related to equivalence classes of projections (see, e.g. \cite{CEI} and \cite{natealin}). 

\begin{corollary}
Let $X$ be a locally compact Hausdorff space that is connected, second countable, and one-dimensional.
Let $A$ be a separable C$^*$-algebra with stable rank one such that $\mathrm{K}_1(I)=0$ for every closed two-sided ideal $I$ of $A$.
Then, an element $[f]\in \Cu(\CC_0(X,A))$ is compact if and only if there exists a compact element $a\in\Cu(A)$ such that $[f(t)]=a$, for all $t\in X$.
\end{corollary}
\begin{proof}
Upon identifying $\Cu(\CC_0(X,A))$ with $\Lsc (X,\Cu(A))$, which we may by Theorem \ref{thm:ese}, we assume that $f\in\Lsc(X,\Cu(A))$ is compact.

Let $t\in X$, and write $f(t)=\sup_{n} f_n(t)$ as in Proposition \ref{prop:charact} where the functions $f_n$ are constant in a neighbourhood $V_t$ of $t$. Then $f$ has a constant value $a_t$ in a neighbourhood of $t$.

Since $f\ll f$ implies $f(t)\ll f(t)$ for all $t\in X$, we have $a_t\ll a_t$. Further, since $X$ is compact, 
we can find a finite cover $(V_{t_i})_{i=1}^k$  with associated compact elements $a_{t_i}$. It is clear that $V_{t_i}\cap V_{t_j}=\varnothing$
if $a_{t_i}\neq a_{t_j}$, so using the connectedness of $X$, we find a unique value $a_{t_i}$, and so $f$ is constant.
\end{proof}


\section{Examples}

We now give some examples of the computation of $\Cu(A)$ for certain C$^*$-algebras.

\subsection*{Recursive Sub-homogeneous algebras}

The class of Recursive Subhomogeneous Algebras (rsh-algebras) defined in \cite{rsh} is the smallest class $\mathcal R$
of C$^\ast$-algebras which contains $\CC(X,\M_{n})$ for all compact Hausdorff spaces $X$ and $n\geq 1$, and which is closed
under isomorphisms and pullbacks of the type

\begin{equation}\label{rsh}\xymatrix{&A\ar[d]^\varphi \\ \CC(X,\M_{n})\ar[r]^\rho & \CC(Y,\M_{n})}
\end{equation}
where $A$ is in $\mathcal R$, $\varphi$ is a unital $\ast$-homomorphism, $Y\subseteq X$ is a closed subspace of $X$ and $\rho$ is 
the restriction map.

If we restrict to the class of rsh-algebras $R$ constructed using compact Hausdorff spaces of dimension at most one, we can 
describe their Cuntz semigroup by an iterated use of Corollary \ref{cor:aquel} as:
\begin{multline*}\Cu(R)\cong \\ \left(\dots\left(\left(\Lsc(X_0,\overline\N)\oplus_{\Lsc(Y_1,\overline\N)}\Lsc(X_1,\overline\N)\right)\oplus_{\Lsc(Y_2,\overline\N)}\Lsc(X_2,\overline\N)
 \right)\dots\right)\oplus_{\Lsc(Y_k,\overline\N)}\Lsc(X_k,\overline\N),
\end{multline*}
where $X_i$ are second countable compact Hausdorff spaces of dimension at most one
and $Y_i\subseteq X_i$ are closed subsets. Note that $\overline \N=\N\cup\{\infty\}$ can be naturally 
identified with the Cuntz semigroup of $\M_n$.

\subsection*{Non-commutative CW-complexes}

Noncommutative CW-complexes introduced by Eilers, Loring and Pedersen in \cite{elp} are a particular case of rsh-algebras. 
A one-dimensional NCCW-complex is the resulting C$^*$-algebra pullback of the following diagram
\begin{equation}
\xymatrix{ & E\ar[d]^\varphi \\ \CC([0,1],F)\ar[r]^\rho & F\oplus F}
\end{equation}
where $E,F$ are finite dimensional C$^*$-algebras, $\varphi$ is an arbitrary $\ast$-homomorphism, and $\rho$ is given by evaluation at $0$ and $1$.
One-dimensional NCCW-complexes cover a large amount of 
C$^\ast$-algebras, including dimension drop algebras, spitting interval algebras, and the building blocks used in the classification of one-parameter continuous fields of AF-algebras  (see \cite{elliott-dadarlat-zhuang}). The classification of inductive limits of one dimensional NCCW-complexes with trivial $\mathrm{K}_1$-group was carried out in \cite{robertnccw} using the functor $\Cu^\sim$ which is related to the functor $\Cu$.

Using Corollary \ref{cor:aquel}, the Cuntz semigroup of a one dimensional NCCW-complex can be computed as the induced 
pullback of ordered semigroups in $\Cu$. We identify the Cuntz semigroup of a finite dimensional C$^*$-algebra with $\overline\N^k$
for some $k$. Since $\varphi\colon E\to F\oplus F$ is any C$^*$-algebra map, we obtain a semigroup map $\Cu(\varphi)\colon \overline\N^r \to \overline\N^{2s}$ which is thus described by a matrix $A\in \M_{2s,r}(\N)$.
Now the map $\Cu(\rho)\colon \Cu(\CC([0,1],F))\to \Cu(F\oplus F)$ is given by evaluation at $0$ and $1$ of lower semicontinuous
functions $f\colon [0,1]\to \overline{\N}^{2s}$. Therefore, the ordered semigroup pullback is isomorphic to 
\[\{(f,b)\in\Lsc([0,1],\overline\N^s)\oplus \overline{\N}^r | \ (f(0),f(1))^t=Ab\}, \] 
which is thus completely determined by the matrix $A$.

\subsection*{Dimension drop algebras over the interval}

Dimension drop algebras are a particular case of non commutative CW-complexes. In fact we will consider a slightly more
general case since we need not restrict to finite dimensional algebras.  

Given two positive integers $p,q$ the dimension
drop algebra is defined as 
\[ Z_{p,q}=\{ f\in \CC([0,1],\M_{p}(\C)\otimes \M_{q}(\C)\mid f(0)\in
\M_{p}(\C)\otimes I_q,\ f(1)\in I_p\otimes \M_q(\C)\}.\]
and can be described as the pullback of the following diagram
\[
\xymatrix{ & \M_p(\C)\oplus \M_q(\C)\ar[d]^\phi \\ 
\CC([0,1],\M_p(\C)\otimes \M_q(\C))\ar[r]^{\ \ \ (\lambda_0,\lambda_1)} & (\M_{p}(\C)\otimes
\M_q(\C))^2 
} 
\]
where $\lambda_i(f)=f(i)$ and $\phi(A,B)=(A\otimes I_q,I_p\otimes B)$.

Identifying $\Cu(\M_r(\C))$ with $\frac{1}{r}\overline\N$ we obtain by Corollary \ref{cor:aquel} 
\[\Cu(Z_{pq})\cong\{ f\in \Lsc([0,1],\frac{1}{pq}\overline\N) \mid \ f(0)\in \frac{1}{p}\overline\N,\  f(1)\in \frac{1}{q}\overline\N\}.\] 

In case $p,q$ are coprime, $Z_{pq}$ is called a prime dimension drop algebra. 
The Jiang-Su algebra $\mathcal Z$ is constructed as an inductive limit of diferent prime dimension drop algebras $Z_{p_nq_n}$
which is simple and has a unique trace. This construction can be slightly simplified using a unique dimension drop algebra
$Z_{pq}$ and a unique morphism $\gamma:Z_{pq}\to Z_{pq}$ but alowing $p,q$ to 
be supernatural numbers of infinite type,
that is to say, $p_i^\infty=p_i$ (see \cite{Rordam-Winter}). The construction for $Z_{pq}$ can also be done using a pullback as before but now
$\M_p(\C)$ should be replaced by the corresponding UHF-algebra.  The Cuntz semigroup of these UHF-algebras can be 
computed using the description given in e.g. \cite{bpt} as 
\begin{align}\label{eq: Cp}
C_p:=\Cu(\lim_{\rightarrow} (\M_n(\C); n\mid p)):= \R^{++}\sqcup \bigcup_{n\mid p}\frac{1}{n}\N \sqcup \{\infty\}.
\end{align}
Hence, observing that $C_p,C_q\subseteq C_{pq}$, we have
\[\Cu(Z_{pq})\cong\{ f\in \Lsc([0,1],C_{pq}) \mid
\ f(0)\in C_p,\  f(1)\in C_q\}.\] 

\subsection*{Dimension drops algebras over a two dimensional space}
Let $X$ be compact Hausdorff space of dimension two such that $\check{\mathrm{H}}^2(X,\Z)=0$, 
and  let $x_1, x_2, \cdots, x_n\in X$. Given supernatural numbers $p_1,p_2,\cdots,p_n$ of infinite type, let us consider the dimension drop algebra
\begin{multline*}
Z_{p_1,p_2,\cdots, p_n}=  \\ \{f\in \mathrm{C}(X,\bigotimes_{i=1}^n\M_{p_i})\mid  f(x_i)\in I_{p_1}\otimes \cdots \otimes I_{p_{i-1}}\otimes \M_{p_i}\otimes I_{p_{i+1}}\otimes \cdots\otimes I_{p_n},\, i=1,2,\cdots, n\}.
\end{multline*} 
This algebra can be described as the pullback of the following diagram:
\[
\xymatrix{ & \bigoplus_{i=1}^n \M_{p_i}\ar[d]^\phi \\ 
\mathrm{C}(X,\bigotimes_{i=1}^n\M_{p_i})\ar[r]^{\lambda\ \ \ \ \ } & \Lsc(Y,\bigotimes_{i=1}^n\M_{p_i} 
)} 
\]
where $Y=\{x_1,x_2,\cdots,x_n\}$, $\lambda(f)=f|_Y$, and 
\[
(\phi(A_1,A_2,\cdots, A_n))_i=I_{p_1}\otimes \cdots \otimes I_{p_{i-1}}\otimes A_i\otimes I_{p_{i+1}}\otimes \cdots\otimes I_{p_n},
\]
for $i=1,2,\cdots,n$. By Corollary \ref{cor: pullbackdimtwo}
\begin{align*}
\Cu(Z_{p_1,p_2,\cdots, p_n})\cong (\bigoplus_{i=1}^nC_{p_i})\oplus_{\Lsc(Y ,C_{p_1p_2\cdots p_n})}\Lsc(X,C_{p_1p_2\cdots p_n}).
\end{align*}
Hence, we have
\begin{align*}
\Cu(Z_{p_1,p_2,\cdots, p_n})\cong\{ f\in \Lsc(X,C_{p_1p_2\cdots p_n}) \mid
\ f(x_i)\in C_{p_i},\,i=1,2,\cdots, n\},
\end{align*}
where $C_p$ is as in \eqref{eq: Cp}. 

\subsection*{Mapping torus of A}

Let $A$ be a C$^*$-algebra and $\phi\colon A\to A$ an automorphism. The mapping torus of the pair $(A,\phi)$  is defined by 
\[ T_{\phi}(A)=\{f\in \CC([0,1],A)\mid f(1)=\phi(f(0))\}\,.\]
Observe that $T_{\phi}(A)$ can be obtained as the pullback in the following diagram
\[
 \xymatrix{T_{\phi}(A)\ar[r]\ar[d] & A\ar[d]^{(\textrm{id},\phi)} \\ \CC([0,1],A)\ar[r]^>>>>>{\rho_{\{0,1\}}} & A\oplus A 
}
\]
Therefore, if $A$ has stable rank one and $K_1(I)=0$ for every ideal in $A$, using Corollary~\ref{cor:aquel} we obtain 
\[ \Cu(T_{\phi}(A))\cong \{f\in\Lsc([0,1],\Cu(A))\mid f(1)=\Cu(\phi)(f(0))\}\,.\]


\section{Semigroups of lower semicontinuous functions}
\label{lsc}

Our aim in this section is to prove that
the set $\Lsc(X,M)$ of lower semicontinuous functions from a compact Hausdorff
space $X$ with finite covering dimension to a countably based semigroup $M$ in Cu, equipped with the
pointwise order and addition, is also a semigroup in Cu. Throughout this section, $X$ will always denote a topological space that is second countable,
compact, and Hausdorff, whence metrizable.


The following lemmas are easy to prove and hence we omit the details.

\begin{lemma}
\label{lem:easy}
If $M$ is a semigroup in Cu, then $f\in\Lsc(X,M)$ if and only if, given $t\in X$ and $a\ll f(t)$,  there exists an open
neighbourhood $U_t$ of $t$ such that $a\ll f(s)$, for all $s\in U_t$.
\end{lemma}

\begin{lemma}\label{lem:easy2}
 Let $f\in \Lsc(X,M)$ and $Y\subseteq X$ a closed set. Then,
 \begin{enumerate}[{\rm (i)}]
 \item The restriction $f|_{Y}$ of $f$ to $Y$ is a function in $\Lsc(Y,M)$.
 \item If $g\in \Lsc(Y,M)$ and $g\leq f|_Y$, then 
\[f_{\downarrow g}(t):=\left\{\begin{array}{rl} g(t) & \text{ if } t\in Y \\ f(t) & \text{ otherwise. }\end{array}\right.\]
is a function in $\Lsc(X,M)$
\end{enumerate}
\end{lemma}

A source of examples of these functions is obtained by the action of the characteristic functions of open subsets of $X$ in $\Lsc(X,M)$. This action is described as follows: given an open set $U\subseteq X$ and a function $f\in\Lsc(X,M)$, $f\cdot \chi_{U}$ is defined by
\begin{align*}
(f\cdot \chi_{U})(x):=
\begin{cases}
f(x)&\text{if }x\in U,\\
0&\text{otherwise}
\end{cases}.
\end{align*}
This function belongs to $\Lsc(X,M)$ by Lemma \ref{lem:easy2}.

\begin{remark}
\label{rem:tontada}
{\rm If $M$ is a semigroup in $\Cu$, and $(a_n)$ is an increasing sequence with $a\ll\sup_n a_n$, then there exists $m$ such that $a\ll a_m$. Indeed, write $\sup_n a_n=\sup_k b_k$, for a rapidly increasing sequence $(b_k)$ in $M$, and find $k$ such that $a\leq b_k$. As $b_{k+1}\leq a_m$ for some $m$, this yields $a\leq b_k\ll b_{k+1}\leq a_m$, so $a\ll a_m$.}
\end{remark}

\begin{lemma} 
\label{lem:estructuramonoide}
Let $M$ be a semigroup in Cu. Then: 
\begin{enumerate}[{\rm (i)}]
\item $\Lsc(X,M)$ endowed with the pointwise addition and order is an ordered semigroup.
\item $\Lsc(X,M)$ is closed under (pointwise) suprema of increasing sequences.
\end{enumerate}
\end{lemma}
\begin{proof}
(i). Let $a\in M$ and $t\in (f+g)^{-1}(a^{\ll})$. Let us write $f(t)=\sup_n f_n$ and $g(t)=\sup_n g_n$ where $(f_n),(g_n)$ are rapidly increasing sequences. Then $(f+g)(t)=f(t)+g(t)=\sup_n (f_n+g_n)$. Since $a\ll (f+g)(t)=f(t)+g(t)=\sup_n (f_n+g_n)$, there exists $N\geq 0$ such that $a\ll f_N+g_N$. Next, since $f_N\ll f(t), g_N\ll g(t)$  and $f,g\in \Lsc(X,M)$, there are open neighbourhoods $U_t$ and $V_t$ of $t$ such that $f_N\ll f(s)$ if $s\in U_t$ and $g_N\ll g(s) $ if $s\in V_t$. Hence, $W_t=U_t\cap V_t$ is an open neighbourhood of $t$, and clearly
\[  
a\ll f_N+g_N\ll f(s)+g(s)=(f+g)(s),
\]
for all $s\in W_t$.
Therefore $W_t\subseteq (f+g)^{-1}(a^{\ll})$, which proves that $(f+g)^{-1}(a^{\ll})$ is open, whence $(f+g)\in \Lsc(X,M)$.

(ii). Let $(f_n)$ be an increasing sequence in $\Lsc(X,M)$, and put $f(t):=\sup_nf_n(t)$. Since $M$ is closed under suprema of increasing sequences, $f$ exists. For any $t\in X$, and $a\ll f(t)=\sup_nf_n(t)$, there exists by Remark \ref{rem:tontada} a number $N$ such that $a\ll f_N(t)$. Lower semicontinuity of $f_N$ now provides a neighbourhood $U_t$ such that $a\ll f_N(s)\leq f(s)$ for all $s\in U_t$. Therefore $f\in \Lsc(X,M)$
and clearly $f=\sup_n f_n$.
\end{proof}

The following proposition provides a characterization of compact containment in these function spaces.


\begin{proposition}
\label{prop:charact}
Let $M$ be a semigroup in Cu. Given  $f,g\in \Lsc(X,M)$ we have $g\ll f$ if and only if 
for every $t\in X$ there are an open neighbourhood $U_t$ of $t$, and  $c_t\in M$ such that $g(s)\leq c_t\ll f(s)$ for all $s\in U_t$.
 \end{proposition}

\begin{proof}
Suppose $g\ll f$. 
Given $t\in X$ let us write $f(t)=\sup_m a_m$  where $(a_m)$ is a rapidly increasing sequence. 
Since $f\in \Lsc(X,M)$ and $X$ is metrizable there exists, for all $m$, an open neighbourhood $V_m$ of $t$ such that 
$a_m\ll f(s)$ for all $s\in \overline V_m$. Moreover, using that $X$ is metrizable we may choose these neighbourhoods to be such that $\{t\}=\bigcap_m V_m$.

Consider the following functions in $\Lsc(X,M)$,
\[
f_m(s):=\begin{cases} a_m  & \text{ if } s\in \overline{V_m} \\ f(s) & \text{ otherwise}\end{cases}
\] 
Using Lemma \ref{lem:easy2} 
we see that $f_m\in \Lsc(X,M)$, and it is clear 
that $f=\sup_m f_m$. Hence, there exists $m_0$ such that $g\leq f_{m_0}$, and this inequality proves the result  by taking $U_t=U_{m_0}$ and $c_t=a_{m_0}$.

Now suppose the condition holds, and consider an increasing sequence $h_n\in \Lsc(X,M)$ such that $f\leq \sup_n h_n$.
For each $t\in X$ there are a neighbourhood $V_t$ of $t$ and $c_t\in M$  such that $g(s)\leq c_t\ll f(s)$,  for all $s\in V_t$. 
Thus, $g(t)\leq c_t\ll f(t)\leq \sup_n h_n(t)$. Hence, there exists $n_t\in \N$  such that $c_t\ll h_{n_{t}}(t)$. 
Lower semicontinuity of $h_{n_t}$ now provides a neighbourhood $U_t$ such that $c_t\ll h_{n_t}(s)$ for all $s\in U_t$. 
Put $W_t=U_t\cap V_t$. Then,
\[
g(s)\leq c_t\ll h_{n_t}(s)\,,
\] 
whenever $s\in W_t$.

By compactness of $X$, there is a finite cover $W_{t_1}, \dots,W_{t_k}$ such that $g(s)\ll h_{n_{t_i}}(s)$ for every $s\in W_{t_i}$.
Since the sequence $(h_n)$ is increasing, there exists $N$ such that $h_{n_{t_i}}\leq h_N$ for all $i$ and thus $g(s)\leq h_N(s)$ for all $s\in X$.
We thus have $g\leq h_N$. This implies that $g\ll f$, as desired.
\end{proof}

With this characterization at hand, we can now prove that addition in $\Lsc(X,M)$ is compatible with compact containment.

\begin{corollary}
\label{cor:additionandwb} Let $M$ be a semigroup in Cu. Let $f_1,f_2,g_1,g_2\in \Lsc(X,M)$ such that $f_1\ll g_1$ and $f_2\ll g_2$. Then, $f_1+f_2\ll g_1+g_2$.  
\end{corollary}
\begin{proof}
Let $t\in X$. As $f_1\ll g_1$ and $f_2\ll g_2$, we obtain by Proposition \ref{prop:charact} neighbourhoods $U_t$ and $V_t$ of $t$, and elements $c_t,d_t\in M$ such that $f_1(s)\leq c_t\ll g_1(s)$ for all $s\in U_t$
and $f_2(s)\leq d_t\ll g_2(s)$ for all $s\in V_t$. Hence, $W_t=U_t\cap V_t$ is an open neighbourhood of $t$ such that for all $s\in W_t$, 
\[(f_1+f_2)(s)=f_1(s)+f_2(s)\leq c_t+d_t\ll g_1(s)+g_2(s) =(g_1+g_2)(s).\]
A second usage of Proposition \ref{prop:charact} yields $f_1+f_2\ll g_1+g_2$. 
\end{proof}

\begin{corollary}\label{cor:car} Let $M$ be a semigroup in Cu,  $f\in \Lsc(X,M)$ and $Y\subseteq X$ a closed set. Then, we have
\begin{enumerate}[{\rm (i)}]
 \item If $a\ll f(s)$ for all $s\in Y$, there exists an open neighbourhood $Y\subseteq U$ such that $a\ll f(s)$ for all $s\in U$. 
Furthermore $a\cdot \chi_{V}\ll f$ for all open sets $V\subseteq Y$.
\item If $g\in \Lsc(X,M)$ and $g\ll f$, then $g|_Y\ll f|_Y$. 
\end{enumerate}
\end{corollary}

\begin{proof}
The first assertion in (i) is a straightforward application of Lemma \ref{lem:easy}. 
Then, since $Y$ is compact, this can be used to find the open sets required by Proposition \ref{prop:charact}. 
Finally, (ii) is a consequence of Proposition \ref{prop:charact}. 
\end{proof}

Recall that in a topological space $X$, the \emph{covering dimension}
is defined as the least $n$ such that any open cover has an open refinement of multiplicity $\leq n+1$, or infinity in case this $n$ does not
exist. Here, a cover $\mathcal U=\{U_\lambda\}_{\lambda\in \Lambda}$ has multiplicity $k$ if every 
$x\in X$ belongs to at most $k$ subsets in $\mathcal U$.

We proceed to show that, in case that $X$ is finite dimensional, every function $f\in \Lsc(X,M)$ can be written as the  supremum of
a directed set of functions. In relevant situations such set can be taken to be a sequence and that will show that $\Lsc(X, M)$
is an object in Cu. 

We first generalize the step functions defined in the case of $X=[0,1]$ to an arbitrary space $X$ (cf. Definition \ref{def:pw01}).

\begin{notation}
Given a family of open sets ${\mathcal U}=\{U_i\}_{i\in \Lambda}$, we write $F_{\mathcal U,t}$, $F'_{\mathcal U,t}$, and $\mathcal{A}_{\mathcal U}$ for the sets:
\begin{align*}
 F_{\mathcal U,t}:=\{ i\in \Lambda\mid t\in U_i\},\quad  F'_{\mathcal U,t}:=\{ i\in \Lambda\mid t\in \overline{U_i}\},\quad \mathcal{A}_{\mathcal U}:=\{ F_{\mathcal U,t},\ F'_{\mathcal U,t}\mid t\in X\}.
\end{align*}
When clear we will omit $\mathcal U$ in the notation. Observe that $\mathcal A_{\mathcal U}$ is a subset of the power set
$\mathcal P(\Lambda)$ hence we order it by inclusion.

\end{notation}

\begin{definition}\label{def:pw} Let $X$ be an $n$-dimensional topological space, $M$ a semigroup in Cu and $f\in \Lsc(X,M)$. A function $g\colon X\to M$ will be termed
a \emph{piecewise characteristic function for $f$} if there are:
\begin{enumerate}[{\rm (i)}]
\item A family of open sets ${\mathcal U}=\{U_i\}_{i=1}^m$ of $X$ such that $\{\overline{U}_i\}_{i=1}^{m}$ has multiplicity $\leq n+1$.
\item An ordered map $\varphi\colon {\mathcal A_{\mathcal U}}\to M$ with $\varphi(\varnothing)=0$
satisfying, for all $t\in X$:
\[g(t)=\varphi(F_{{\mathcal U},t})\leq \varphi(F'_{{\mathcal U},t})\ll f(t).\]
\end{enumerate}
We will use the notation $g:=\chi({\mathcal U},\varphi)$ to refer to such a function. The set of all piecewise characteristic functions for $f$
will be denoted by $\chi(f)$.
\end{definition}


\begin{lemma}\label{lem:proppw} If $f\in \Lsc(X,M)$ and $g$ is a piecewise characteristic function for $f$, then $g\in \Lsc(X,M)$ and $g\ll f$.\end{lemma}

\begin{proof}
Consider $g=\chi(\{U_i\}_{i=1}^m,\varphi)$ as in Definition \ref{def:pw} with $g(t)=\varphi(F_t)\leq \varphi(F'_t)\ll f(t)$. 

Given $a\ll g(t)$, observe that $t\in V_t:=\bigcap_{i\in F_{t}}U_i$ which is an open neighbourhood of $t$.
 Then, for all $s\in V_t$ we have $F_s\supseteq F_t$ and therefore $a\ll g(t)=\varphi(F_t)\leq \varphi(F_s)=g(s)$ for all $s\in V_t$.   
By Proposition \ref{lem:easy}, this proves that $g\in \Lsc(X,M)$. 

Now observe that given $t\in X$, there is a neighbourhood $V_t$ of $t$ that only intersects the sets $U_i$ with $i\in F'_t$. 
Hence $F_s\subseteq F'_t$ for all $s\in V_t$. 
Therefore, $g(s)=\varphi(F_s)\leq \varphi(F'_t)\ll f(t)$ for all $s\in V_t$. By lower semicontinuity 
of $f$ we can choose this neighbourhood in such a way that $g(s)\leq \varphi(F'_t) \ll f(s)$ for all $s\in V_t$. 
This implies by Proposition \ref{prop:charact} that $g\ll f$.
\end{proof}

\begin{lemma}
\label{lem:sup}
Let $M$ be a semigroup in Cu and $f\in \Lsc(X,M)$. Then 
\[
f=\sup\{ g\mid g\in \chi(f)\}\,.
\] 
\end{lemma}

\begin{proof}
Given $t\in X$, let us write $f(t)=\sup_na_n$ where $a_n\ll a_{n+1}$.  Given $n$, lower semicontinuity of $f$ provides us with a neighbourhood $U'_n$
such that $a_n \ll f(s)$ for all $s\in U'_n$. Since $X$ is normal, we can find an open neighbourhood $U_n$ such that $\overline U_n\subseteq U'_n$.
Hence, 
\[
g_n=a_n\cdot \chi_{U_n}=\chi(\mathcal U_n,\varphi_n),
\]
with $\mathcal U_n=\{U_n\}$, $\varphi(\varnothing)=0$ and $\varphi(\{n\})=a_n$, 
is a piecewise characteristic function for $f$ with $g_n(t)=a_n$.
If now $h\in \Lsc(X,M)$ is such that $h\geq g$ for all $g\in\chi(f)$, then in particular $h(t)\geq g_n(t)$.
Thus $h(t)\geq a_n$ for each $n$, that is, $h(t)\geq f(t)$. Since the supremum in $\Lsc(X,M)$ is the pointwise supremum, we obtain that 
$f$ is the supremum of its piecewise characteristic functions. 
\end{proof}

The previous lemma describes each element $f$ in $\Lsc(X,M)$ as the supremum of piecewise characteristic functions that are compactly contained in $f$. But to prove that $\Lsc(X,M)$ is an object in Cu, we need to write $f$ as the supremum of a \emph{sequence}. In order to prove this, we will first show that the set of piecewise characteristic
functions form a directed set, and that it can be chosen to be countable if furthermore $M$ is countably based.

To this end, we shall use induction on the dimension of the space. The key of the inductive step is encoded in the following lemma.

\begin{lemma}\label{lem:pw}
Let $Y\subset X$ be a closed set, let $f\in \Lsc(X,M)$, and let $g\in \Lsc(Y,M)$ be a piecewise characteristic function for $f|_Y$. 
Then, there exists a piecewise characteristic function $h$ for $f$ such that $g\leq h|_{Y}$. 

Moreover, if $g=\chi(\mathcal W,\phi)$, then for every $\delta>0$ $h$ can be constructed as 
$h=\chi(\mathcal W^\epsilon,\phi^{\epsilon})$, where $0<\epsilon< \delta$ and $\mathcal W^\epsilon$ 
consists of the $\epsilon$-neighbourhoods of the elements of $\mathcal W$, and $\phi^{\epsilon}$ is a restriction of $\phi$.

\end{lemma}

\begin{proof}
Suppose $g=\chi({\mathcal W},\phi)$ where ${\mathcal W}=\{W_i\}_{i=1}^{k}$. For every $\epsilon>0$ and $i=1,\dots,k$, we denote by $W_i^\epsilon$
the $\epsilon$-neighbourhood of $W_i$. That is, $W_i^\epsilon=\{x\in X\mid d(x,\overline{W_i})<\epsilon\}$. 
We claim that $\epsilon$ can be chosen in such a way that 
${\mathcal A_{\mathcal W}}\supseteq {\mathcal A}_{\mathcal W^\epsilon}$.

Observe that for all $\epsilon>0$ and $t\in X$, we have $F'_{\mathcal W^\epsilon,t}\supseteq F_{\mathcal W^\epsilon,t}\supseteq F'_{\mathcal W,t}$. For each $j$ such that $t\not\in \overline{W_j}$, there exists $\delta_j>0$, such that $t\not\in \overline{W_j^\delta}$ for all $\delta\leq\delta_j$.
Hence, for each $t$ there exists $\epsilon_t>0$ such that $F'_{\mathcal W^{\epsilon_t},t}=F_{\mathcal W^{\epsilon_t},t}=F'_{\mathcal W,t}$ and hence there is 
a neighbourhood $V_t$ of $t$ such that $F'_{\mathcal W^{\epsilon_t},s}=F_{\mathcal W^{\epsilon_t},s}=F'_{\mathcal W,t}$ for all $s\in V_t$. Since $X$
is compact, we can find a finite number of such neighbourhoods $V_{t_i}$ covering $X$ and hence there is $0<\epsilon<\epsilon_{t_i}$ for which we will have 
 ${\mathcal A_{\mathcal W}}\supseteq {\mathcal A}_{\mathcal W^\epsilon}$. In this situation, $\{ \overline{W^\epsilon_i}\}$ has the same multiplicity as $\{\overline{W_i}\}$.

Now, for each $F\in {\mathcal A}_{\mathcal W}$ and $t\in \bigcap_{i\in F}\overline{W_i}$ we have $F\subseteq F'_t$, and therefore $\phi(F)\leq \phi(F'_t)\ll f(t)$. 
Hence $\phi(F)\ll f(t)$ for all $t\in \bigcap_{i\in F}\overline{W_i}=W_F$. 
Since $W_F$ is closed in $X$, by Corollary \ref{cor:car} there is an open neighbourhood $U_F$ of $W_F$ (in $X$)
such that $\phi(F)\ll f(t)$. Thus we can make $\epsilon>0$ smaller if necessary so that $\bigcap_{i\in F} \overline{W^{\epsilon}_i}\subset U_F$, and in such a way that this is true for all $F\in {\mathcal A}_{\mathcal W}$. Then, 
\begin{equation}\label{pwh}
\phi(F)\ll f(t) \text{ for all }t\in\bigcap_{i\in F} \overline{W^{\epsilon}_i}.
\end{equation}

Now, since ${\mathcal A_{\mathcal W}}\supseteq {\mathcal A}_{\mathcal W^\epsilon}$, we consider $h=\chi({\mathcal W}^{\epsilon},\phi^\epsilon)$ where
$\phi^\epsilon=\phi|_{\mathcal A_{\mathcal W^{\epsilon}}}$. 

Given $t\in X$ we have $t\in\bigcap_{i \in F'_{{\mathcal W}^\epsilon,t}}\overline{W_i^{\epsilon}}$. Thus, by (\ref{pwh}) we have 
\[h(t)=\phi^{\epsilon}(F_{\mathcal W^\epsilon,t})\leq \phi^{\epsilon} (F'_{\mathcal W^\epsilon,t})\ll f(t),\]
proving that $h$ is a piecewise characteristic function for $f$.

Finally, it is clear that given $t\in Y$, $F_{\mathcal W,t}\subseteq F_{{\mathcal W^{\epsilon}},t}\in  {\mathcal A}_{\mathcal W^\epsilon}
\subseteq {\mathcal A_{\mathcal W}}$.  Hence, 
$g(t)=\phi(F_{\mathcal W,t})\leq \phi(F_{\mathcal W^{\epsilon},t})= \phi^{\epsilon}(F_{\mathcal W^{\epsilon},t})=h(t)$,
impliying $h|_Y\geq g$.

The last assertion is clear by the construction of $h$.
\end{proof}

\begin{proposition}
\label{prop:updirected}
Let $M$ be a semigroup in Cu. Let $f$, $g_1$ and $g_2 \in \Lsc(X,M)$ and suppose that  $g_1,g_2\ll f$.
Then, there exists $h\in\chi(f)$ such that $g_1,g_2\ll h$. In particular,
$\chi(f)$ is an upwards directed set. 
\end{proposition}

\begin{proof}

Let $\eta>0$. We will prove, by induction on the dimension of $X$, that there exists an open cover $\mathcal U$ of $X$ and $h=\chi(\mathcal U,\varphi)$ such that 
$g_1,g_2\ll h\ll f$ and each open set $U_i\in\mathcal U$ is contained in an $\eta$-ball.

 By Proposition \ref{prop:charact}, for any $t\in X$ we can find an open neighbourhood $V'_t$ of $t$ and elements $a_t,b_t\in M$ such that $g_1(s)\leq a_t\ll f(t)$ and
$g_2(s)\leq b_t\ll f(t)$ for all $s\in V'_t$. We may further assume that each $V'_t$ is a $\delta_t$-ball with center in $t$ for some $0<\delta_t<\eta$.
As elements compactly contained in $f(t)$ form a directed set, there exists $c_t\in M$ such that 
$a_t,b_t\ll c_t\ll f(t)$. Now, since $f$ is also lower semicontinuous, we can choose the previous neighbourhoods in such a way that $c_t\ll f(s)$ for all $s\in V'_t$,
and therefore $g_1(s),g_2(s)\ll c_t \ll f(s)$ for all $s\in V'_t$. Now let $V_t$ be a $\delta_t/2$-ball with center $t$, so that
\[ g_1(s),g_2(s)\ll c_t \ll f(s) \text{ for all }s\in \overline{V_t}.\]

By compactness, there exists a finite cover for $X$ of the form ${\mathcal V}=\{V_{t_i}\}_{i=1}^{k}$. 

In case $X$ has dimension 0 this cover has a finite disjoint refinement, which means we can assume $\mathcal V$ is a finite cover of disjoint clopen sets. 
Hence, ${\mathcal A_{\mathcal V}}=\{\{1\},\dots,\{k\}\}$, and we can
consider the piecewise characteristic
function $h:=\chi({\mathcal V},\varphi)$, where 
$\varphi(\{i\})=c_{t_i}$. By construction each $V_{t_i}$ is contained in an $\eta$-ball, and
 it is not difficult to check that $g_1,g_2\ll h$.

Now suppose $\dim X=n\geq 1$ and that the result holds true for spaces of smaller dimension. Retain the construction of $V_{t_i}, V'_{t_i}, c_{t_i}$ as before. Using \cite[4.2.2]{engelkin}, we may assume without loss of generality that the boundary, $Y=\bigcup_{i=1}^k\bd({V_{t_i}})$ has dimension at most $n-1$.  Let $\delta=\min\{\delta_{t_i}/3\}$ and so 
we have $V_{t_i}^\delta\ll V'_{t_i}$ for all $i=1,\dots,k$, where $V_{t_i}^\delta$ is a $\delta$-neighbourhood of $V_{t_i}$.

Since $Y\subseteq X$ is a closed set, $f|_Y\in \Lsc(Y,M)$ (by Lemma \ref{lem:easy2}). 

For all $i=1,\dots,k$, put $V^Y_{t_i}=Y\cap V^\delta_{t_i}$ and we have $g_1(s),g_2(s)\leq c_{t_i}\leq f(s)$ for all $s\in \overline{V^Y_{t_{i}}}$. 
Hence we have $c_{t_i}\chi_{V^Y_{t_i}} \leq f|_Y$ for all $i$ by Corollary \ref{cor:car} (i).
By induction, there exists an open cover ${\mathcal W}=\{W_{j}\}_{j=1}^{r}$  of $Y$, with each $W_j$ contained in a $\delta/3$-ball, and a piecewise characteristic function $g_Y=\chi({\mathcal W},\phi)$ for $f|_Y$ such that 
\begin{equation}\label{updir} c_{t_i}\chi_{V^Y_{t_i}}\ll g_Y\ll f|_Y \ \text{ for all } i=1,\dots,k.\end{equation} 
Observe that whenever $\overline{W_j}\cap \overline{V_{t_i}}\neq \varnothing$, we have $\overline{W_j}\subseteq V_{t_i}^{\delta}\subseteq V_{t_i}'$. 
Now we use Lemma \ref{lem:pw} to obtain a piecewise characteristic function $h'=\chi(\mathcal W^\epsilon,\phi)$ for $f$ such that $g_Y\leq h'|_Y$. 
Decreasing $\epsilon$ if necessary, we can further assume that each $W_j^\epsilon$ is contained in a $\delta/2$-ball, and hence $\overline{W^\epsilon_{j}}\subseteq \overline{V_{t_i}^\delta}\subseteq V'_{t_i}$ whenever $\overline{W^\epsilon_j}\cap \overline{V_{t_i}}\neq \varnothing$.

Let $Y^{\epsilon}=\bigcup_{j=1}^{r} W^{\epsilon}_j$.

Put $U_i=W^{\epsilon}_i$ for $i=1,\dots,r$. For each 
$F_j\in \mathcal A_\mathcal V =\{F_1,\dots,F_{r'}\}$, 
let  $U_{r+j}$ be an $\epsilon/3-$neigh\-bour\-hood of $\bigcap_{i\in F_j}V_{t_i}\setminus (Y^\epsilon\cup (\bigcup_{k\not\in F_j}V_{t_k}))$. 
Observe that $U_{r+j}\subseteq \bigcap_{i\in F_j}V_{t_i}$, and that $U_{r+l}\cap U_{r+l'}=\varnothing$ for all $l\neq l'$.

Next consider the cover ${\mathcal U}=\{U_1,\dots,U_{r},U_{r+1},\dots, U_{r+r'}\}$.
Observe that if $F\in {\mathcal A}_{\mathcal U}$, then either $F\in {\mathcal A}_{\mathcal W}$ or $F=F'\cup \{r+l\}$ 
for some $F'\in {\mathcal A}_{\mathcal W}$, or else $F=\{r+l\}$. Since $\overline{\mathcal W}$ has multiplicity at least $n$, we see that $\overline{\mathcal U}$ has multiplicity 
at least $n+1$.

Let $\varphi\colon {\mathcal A}_{\mathcal U}\to M$ be defined by

\[ 
\varphi(F)=\left\{ \begin{array}{lr} \phi(F) & \text{ if } F\in {\mathcal A}_{\mathcal W} \\
                      \phi(F') & \text{ if } F=F'\cup \{r+l\}\text{ for some } l\geq 1,\text{ and }  F'\in {\mathcal A}_{\mathcal W} \\
c_{t_{i_1}}\quad & \text{ if } F=\{r+l\},\text{ and } F_l=\{i_1<\dots<i_{k_l}\}  \end{array}\right.
\]
and let $h=\chi({\mathcal U},\varphi)$. 
We claim that $h$ is a piecewise characteristic function for $f$ such that $g_1,g_2\ll h$.

(i) $\varphi$ is an ordered map.

By definition of $A_\mathcal U$, we only have to consider the following cases:
\[ F'\subseteq F'',\quad F'\cup \{r+l\}\subseteq F''\cup \{r+l\},\quad    F'\subseteq F''\cup \{r+l\},\quad \{r+l\}\subseteq F'\cup \{r+l\},\]
where $F'\subseteq F''\in \mathcal A_{\mathcal W}$. By definition of $\varphi$ and since $\phi$ is an ordered map, 
the only non trivial case is $\{r+l\}\subseteq F'\cup \{r+l\}$.

Then, suppose $F_l=\{i_1<\dots<i_{k_l}\}$, and $F'_{\mathcal U,t}=F'\cup \{r+l\}$. 
This means $t\in\left( \bigcap_{j\in F'} \overline{W^\epsilon_j}\right)\cap \left(\bigcap_{j=1}^{k_l}\overline{V_{t_{i_j}}}\right)$. 
There exists $t_{Y}\in Y$ such that $F'=F'_{\mathcal W,t_Y}$. Given $j\in F'$, since $t\in \overline{W^{\epsilon}_j}\cap \overline{V_{t_{i_1}}}\neq \varnothing$, we have 
$\overline{W^\epsilon_j}\subseteq V'_{t_{i_1}}$, whence $t_Y\in V^Y_{t_{i_1}}$. Therefore, using (\ref{updir}), we have 

 \begin{align*}
 \varphi(\{r+l\}) & =c_{t_{i_1}}=c_{t_{i_1}}\cdot\chi_{V_{t_{i_1}}^Y}(t_Y) \\ & \ll g_Y(t_Y)\leq h'(t_Y)=\phi(F_{\mathcal W^\epsilon,t_Y}) \\ &\leq 
\phi(F'_{\mathcal W^\epsilon,t_Y})=\phi(F')=\varphi(F'\cup \{r+l\}),
\end{align*}
proving that $\varphi$ is an ordered map.

(ii)  $h$ is a piecewise characteristic function for $f$.

Let $t\in X$. If $t\in \overline{Y^{\epsilon}}$, then $F'_{\mathcal U,t}$ equals $F'_{\mathcal W^{\epsilon},t}$ or $F'_{\mathcal W^{\epsilon},t}\cup \{r+l\}$. Hence 
$\varphi(F'_{\mathcal U,t})=\phi(F'_{\mathcal W^\epsilon,t})\ll f(t)$ since $h'$ is a characteristic function for $f$. 

Otherwise, if $t\not\in \overline{Y^{\epsilon}}$, then $t\in U_{r+l}$ for some $l\geq 1$, and 
$F'_{\mathcal U,t}=F_{\mathcal U,t}=\{r+l\}$, where $F_l=\{i_1<\dots <i_l\}$. 
In particular,  $t\in V_{t_1}$, therefore $c_{t_1}\ll f(t)$, and we have
\[ \varphi(F'_{\mathcal U,t})=\varphi(\{r+l\})=c_{t_1}\ll f(t).\]
This proves that $h\in \chi(f)$.

(iii) $g_1\ll h$ and $g_2\ll h$.

Let $t\in X$. If $t\in W^\epsilon_j$ for some $j$, since $W^{\epsilon}_j\subseteq V'_{t_i}$ for some $i$, 
we have \[g_1(s),g_2(s)\ll c_{t_i}\ll g_Y(s)\ll h(s)\] for all $s\in W^{\epsilon}_j$.
Otherwise if $t\not\in Y^\epsilon$, we have $t\in U_{r+l}\subseteq V_{t_{i_1}}$, for some $l$ and $i_1$ the first element in $F_l$. Hence for all $s\in U_{r+l}$, 
\[g_1(s),g_2(s)\ll c_{t_1}=\varphi(\{r+l\})\leq \varphi(F_{\mathcal U,s})=h(s),\] 
proving that $g_1,g_2\ll h$. 
\end{proof}

\begin{proposition}
\label{lem:supseq}
Let $X$ be a finite dimensional topological space, let $M$ be an object in Cu, and let $f\in\Lsc(X,M)$. If $M$ is countably based, then $f$ is the supremum of a rapidly increasing sequence of elements from $\chi (f)$.
\end{proposition}
\begin{proof}
For any function $h\colon X\to M$, put
\[
U_h=\{(t,a)\in X\times M\mid a\ll h(t)\}\,,
\]
which is an open set when $h$ is lower semicontinuous. We know from Lemma \ref{lem:sup} that $f=\sup\{g\mid g\in\chi(f)\}$, 
whence $U_f=\bigcup_{g\in\chi (f)}U_g$. Since by assumption $M$ is countably based, $X\times M$ has a countable basis, 
and so does $U_f$. As these spaces satisfy the Lindel\"of property, there is a sequence $(g_n)$ in $\chi (f)$ such that 
$U_f=\bigcup U_{g_n}$. The sequence $(g_n)$ may be taken to be rapidly increasing by virtue of Proposition \ref{prop:updirected}. 
This implies that $f=\sup g_n$, as was to be shown.
\end{proof}

Assembling the results above, we obtain the following: 

\begin{theorem} 
\label{thm:lscinCu}
Let $X$ be a second countable finite dimensional compact Hausdorff topological space, and let $M$ be an object in Cu. If $M$ is countably based, then $\Lsc(X,M)$ (with the pointwise order and addition)
is also a semigroup in Cu. 
\end{theorem}

\begin{proof}
Combine Lemma \ref{lem:estructuramonoide}, Corollary \ref{cor:additionandwb} and Proposition \ref{lem:supseq}.
\end{proof}

We now proceed to study some functorial properties of $\Lsc$.

\begin{lemma}\label{lem:lscfunct}
 Let $X,Y$ be finite dimensional second countable compact Hausdorff spaces and $M,N$ be countably based semigroups in Cu. 
\begin{enumerate}[{\rm (i)}]
 \item If $f\colon X\to Y$ is a (proper) continuous map
then,
\[ \begin{array}{rccc}\Lsc(f,M)\colon & \Lsc(Y,M)& \longrightarrow &  \Lsc(X,M) \\ &
g & \longmapsto & g\circ f\end{array} \]
is an a map in Cu.
\item  If $\alpha\colon M\to N$ is a map in Cu, then 
\[\begin{array}{rccc}\Lsc(X,\alpha) \colon &  \Lsc(X,M)& \longrightarrow & \Lsc(X,N) \\ &
  g & \longmapsto & \alpha\circ g,\end{array}\]
is also a map in Cu.
\end{enumerate}
\end{lemma}

\begin{proof}
(i) It is easy to see that the map is well defined since $f$ is continuous.  It also preserves order, addition and suprema,
 since those are defined pointwise. 
To prove preservation of compact containment, we will use Proposition \ref{prop:charact}.
Assume $g_1\ll g_2$.  
For each $t\in X$, since $f(t)\in Y$ and $g_1\ll g_2$, there exists $c_{f(t)}\in M$ 
and $V_{f(t)}$ an open neighbourhood of $f(t)$ such that $g_1(s)\leq c_{f(t)}\ll g_2(s)$ for all $s\in V_{f(t)}$. Hence, 
$(g_1\circ f)(s) \leq c_{f(t)} \ll  (g_2\circ f)(s)$ for all $s\in f^{-1}(V_{f(t)})$ which is an open neighbourhood of $t$. Therefore
$g_1\circ f \ll g_2\circ f$.

(ii) Let us first see that $\Lsc(X,\alpha)$ is well defined, which is to say, 
$\Lsc(X,\alpha)(g)\in \Lsc(X,N)$ for all $g\in \Lsc(X,M)$. Let $g\in \Lsc(X,M)$ be fixed, and let $x\in X$
and $a\in N$ be such that  $a\ll \Lsc(X,\alpha)(g)(x)$. Choose a rapidly increasing sequence $(b_n)_{n\in \N}$
in $M$ such that  $\sup_nb_n=f(x)$. Since $a\ll \Lsc(X,\alpha)(g)(x)=\sup_n\alpha(b_n)$ there exists $n_0\ge 1$
such that  $a\le \alpha(b_{n_0})$. Since $g$ is lower semicontinuous and $b_{n_0}\ll g(x)$ there exists a
neighbourhood  $U$ of $x$ such that $b_{n_0}\ll g(y)$, for all $y\in U$. Hence, it follows that 

 \[
 a\le \alpha(b_{n_0})\ll \alpha(g(y))=\Lsc(X,\alpha)(g)(y),
 \] 
 for all $y\in U$. Since $x$ and $a$ are arbitrary this implies that $\Lsc(X,\alpha)(g)\in \Lsc(X,N)$.
 
 It is clear that $\Lsc(X,\alpha)$ preserves the zero element, the order, and suprema of increasing sequences
 since $\alpha$ does. To complete the proof let us show that $\Lsc(X,\alpha)$ preserves compact containment. 
Let $g,h\in \Lsc(X,M)$ be such that $g\ll h$. By Proposition \ref{prop:charact} for all $x\in X$ there exist
 $c_x\in M$ and a neighbourhood $U$ of $x$ such that $g(y)\le c_x\ll h(y)$, for all $y\in U$. It follows that 
 \[
 \Lsc(X,\alpha)(g)(y)=\alpha(g(y))\le \alpha(c_x)\ll \alpha(h(y))=\Lsc(X,\alpha)(h)(y),
 \]
 for all $y\in U$. Therefore, by applying Proposition \ref{prop:charact} again we conclude that $\Lsc(X,\alpha)(g)\ll \Lsc(X,\alpha)(h)$.
\end{proof}

As a consequence of Lemma \ref{lem:lscfunct} and Theorem \ref{thm:lscinCu} we obtain

\begin{theorem}\label{thm:functor}
Let $X$ be a finite dimensional second countable compact Hausdorff space and let $M$ be a countably based semigroup in Cu.
Then $\Lsc(\cdot,M)$ defines a contravariant functor from the category of 
finite dimensional, second countable, compact Hausdorff topological spaces to Cu, and $\Lsc(X,-)$ defines a covariant
functor from the category of countably based semigroups in Cu to Cu.
\end{theorem}


This functor is seen to be sequentially continuous in the following cases:

\begin{proposition}\label{prop:limits}
\begin{enumerate}[{\rm (i)}]
\item Let $(X_i,\mu_{i,j})_{i, j\in \N}$ be an inverse system of compact Hausdorff one-di\-men\-sio\-nal
spaces with surjective maps $\mu_{i,j}\colon X_j\to X_i$ for $j\geq i$.
If $M$ is a countably based semigroup in $\Cu$, then 
\[\Lsc(\varprojlim(X_i,\mu_{i,j}),M)=\varinjlim(\Lsc(X_i,M),\Lsc(\mu_{i,j},M))\,.\]

\item  Let $(M_i,\alpha_{i,j})_ {i,j\in\N}$ be a directed system of countably based semigroups in Cu.
If $X$ is a second countable, compact, Hausdorff space, then, 
 \[ \Lsc(\varinjlim (M_i,\alpha_{i,j}),M)=\varinjlim (\Lsc(X,M_i),\Lsc(X,\alpha_{i,j}))\,. \]
\end{enumerate}
\end{proposition}

\begin{proof}

Recall that the category Cu has limits of inductive sequences (see e.g.  \cite{CEI}),
and that given a directed set $(S_i,\gamma_{i,j})$, a semigroup $S$ with 
maps $\gamma_i:S_i\to S$  is the directed limit $\varinjlim (S_i,\gamma_{i,j})$ if and only if the following two conditions are satisfied:
\begin{enumerate}[{\rm (a)}]
 \item For all $s\in S$, $s=\sup_i\gamma_i(s_i)$ for some $s_i\in S_i$.
 \item If $s_i \in S_i$ and $s_j \in S_j$ are such that $\gamma_i(s_i)\leq \gamma_j(s_j)$, then,
 for all $x\ll s_i$ there exists $k\in \N$ such that $\gamma_{i,k}(x)\ll \gamma_{j,k}(s_j)$. 
\end{enumerate}

(i). Let $X=\varprojlim(X_i,\mu_{i,j})_{i,j\in \N}$ with $\mu_i\colon X\to X_i$ the canonical maps. 
Note that $X$ is also a  one-dimensional compact Hausdorff space and that the canonical maps $\mu_i$ are all surjective. 
The open sets in $X$ can be described as $\bigcup_{i\in \N} \mu_i^{-1}(U_i)$ where each $U_i$ is an open set in $X_i$. Furthermore, the 
$U_i$'s can be chosen in such a way that $\mu_i^{-1}(U_i)\subseteq \mu_{j}^{-1}(U_j)$ if $i\leq j$ (see, e.g. \cite[Propositions 1-7.1 and 1-7.5]{pears}). 

By Theorem \ref{thm:lscinCu},  both $\Lsc(X_i,M)$ and $\Lsc(X,M)$ are objects in Cu. Therefore, using Theorem \ref{thm:functor},
we obtain a directed system $(\Lsc(X_i,M),\rho_{i,j})_{i,j\in \N}$ in Cu, with maps 
$\rho_i\colon \Lsc(X_i,M)\to \Lsc(X,M)$ given by $\rho_i(f)=f\mu_i$, and $\rho_{i,j}\colon \Lsc(X_i,M)\to \Lsc(X_j,M)$ given by
$\rho_{i,j}(g)=g\mu_{i,j}$ whenever $i\leq j$. 

To prove condition (a) above, since $f\in\Lsc(X,M)$ can be described as the supremum of a sequence in $\chi(f)$,
we may assume that $f$ itself is a piecewise characteristic function. Hence suppose $f=\chi(\mathcal U,\phi)$ where 
$\mathcal U=\{U_j\}_{j=1}^r$ is a family of open sets such that $\overline{\mathcal U}=\{\overline U_j\}_{j=1}^r$ has multiplicity
at most one, and $\phi:\mathcal A_{\mathcal U}\to M$ is an ordered map. Let us write each $U_j$ as $\bigcup_{i\in \N}\mu_i^{-1}(U_{j,i})$ for some open sets $U_{j,i}$
in $X_i$ and such that $\mu_i^{-1}(U_{j,i_1})\subseteq \mu_i^{-1}(U_{j,i_2})$ if $i_1\leq i_2$.
Now, for each $k\geq 1$ we consider the family of open sets $\mathcal U_{k}=\{\mu_k^{-1}(U_{j,k})\}_{j=1}^{r}$. We observe that 
both $\mathcal U_k$ and $\overline{\mathcal U_k}$ have multiplicity at most one. 

For each $k\geq 1$ we consider the map 
\[ \begin{array}{ccl}\phi_k\colon \{ F_{\mathcal U_k,t}\ \mid \ t\in X\}  & \rightarrow & M \\ F & \mapsto &  \left\{ \begin{array}{lr} \phi(F) & \text { if }   F\in\mathcal A_{\mathcal U}\\
                                                                      0 & \text{ otherwise.}
                                                                     \end{array}\right.\end{array}\]
and $f_k\colon X\to M$ defined by $f_k(t)=\phi_k(F_{\mathcal U_k,t})$. Following the proof of Lemma~\ref{lem:proppw}, if $\phi_k$
is an ordered map then $f_k$ is lower semicontinuous. But since $\phi$ is already an ordered map we only need to check the case 
$F_1\subset F_2$ with $F_2\not \in \mathcal A_{\mathcal U}$ and $F_1\in \mathcal A_{\mathcal U}$. 
In this case, since $F_2=F_{\mathcal U_k,t}$ for some $t$, there exists some $F_3\in \mathcal A_{\mathcal U}$
such that $F_2\subset F_3$. But $\overline{\mathcal U}$ has multiplicity at most one and therefore subsets in $\mathcal A_{\mathcal U}$
have at most two elements. Therefore  $F_3=\{i,j\}$,  $F_2=\{i\}$ and $F_1=\varnothing$. Thus we obtain $\phi_k(F_1)=\phi_k(\varnothing)=0\leq \phi_k(F_2)$.

For all $t\in X$, there exists $k\in \N$ such that $F_{\mathcal U_k,t}=F_{\mathcal U,t}$. Hence, for all $t\in X$ there exists $k$
such that $f(t)=\phi(F_{\mathcal U,t})=\phi_{k}(F_{\mathcal U_k,t})=f_k(t)$, and since supremum in $\Lsc(X,M)$ is the pointwise supremum,
we obtain $f=\sup_{k\in \N} f_k$. Now we consider $\mathcal V_k=\{U_{j,k}\}_{j=1}^r$. Then $t\in \mu_k^{-1}(U_{j,k})$ if and only if $\mu_k(t)\in U_{j,k}$. It follows that
$F_{\mathcal V_k,\mu_k(t)}=F_{\mathcal U_k,t}$, and hence, taking $g_k(s)=\phi_k(F_{\mathcal V_k,s})$, we have $g_k\in \Lsc(X_k,M)$ 
and $f_k=\rho_k(g_k)$ since $f_k(t)=\phi_k(F_{\mathcal U_k,t})=\phi_k(F_{\mathcal V_k,\mu_k(t)})=g_k\mu_k(t)$.

We now prove condition (b). Suppose that for some $i\leq j$,  $g_i\in\Lsc(X_i,M)$ and $g_j\in \Lsc(X_j,M)$ are such that $\rho_i(g_i)\leq \rho_j(g_j)$.
Let $h\ll g_i$ and hence $\rho_i(h)\ll \rho_i(g_i)\leq \rho_j(g_j)$ which implies $\rho_i(h)\ll \rho_j(g_j)$. 
Using Proposition \ref{prop:charact}, there is for each $t\in X$ an open neighbourhood $V_t$ of $t$ and
$c_t\in M$ such that \[ \rho_i(h)(s)\ll c_t \ll \rho_j(g_j)(s)\ \text{ for all } s\in V_t.\]
We may assume $V_t=\mu_{i_t}^{-1}(U_{t})$ for some open set $U_{t}\subseteq X_{i_t}$, where $i_t\geq i,j$. 
Then, for all $s\in V_t$,
\[\rho_i(h)(s)\ll c_t \ll \rho_j(g_j)(s)\ \Leftrightarrow  h\mu_i(s)\ll c_t \ll g_j \mu_j(s) 
\Leftrightarrow   h\mu_{i_t,i}\mu_{i_t}(s) \ll c_t \ll g_j\mu_{i_t,j}\mu_{i_t}(s)\,,
\]
and hence, for all $s'\in U_t$,
\begin{equation}
  h\mu_{i_t,i}(s')\ll c_t \ll g_j\mu_{i_t,j}(s') 
\Leftrightarrow   \rho_{i_t,i}h(s')\ll c_t \ll  \rho_{i_t,j}g_j(s') \,. \label{wb}
\end{equation}

Using compactness of $X$, we obtain a finite number of open sets $V_{t_i}$ such that $X=\bigcup_{i=1}^r V_{t_i}$, and we can moreover choose $i_{t_1}=\dots=i_{t_r}=k$
for some $k\in \N$. Since $\mu_k$ is surjective then $X_k=U_{t_1}\cup\dots \cup U_{t_r}$ which, together with
(\ref{wb}) and Proposition \ref{prop:charact} proves $\rho_{k,i}(h)\ll \rho_{k,j}(g_j)$.

(ii). Let $M=\varinjlim (M_i,\alpha_{i,j})$ with $\alpha_i\colon M_i\to M$ the canonical maps. 
Let $f,g\in \Lsc(X,M_i)$ be such that $\Lsc(X,\alpha_{i})(f)\le \Lsc(X, \alpha_{i})(g)$, and let $h\ll f$. By Proposition \ref{prop:charact} for all $x\in X$ there exist $c_x\in M$ and a neighbourhood $U$ of $x$ such that $h(y)\le c_x\ll f(y)$, for all $y\in U$. We have 
 \[
  \alpha_{i}(f(x))=\Lsc(X,\alpha_{i})(f)(x)\le \Lsc(X, \alpha_{i})(g)(x)=\alpha_{i}(g(x)).
 \]
 Hence, since $c_x\ll f(x)$ by the characterization of inductive limits in the category Cu that there exists $j\ge i$ such that 
 \[
 \alpha_{i,j}(c_x)\ll \alpha_{i,j}(g(x))=\Lsc(X, \alpha_{i,j})(g)(x).
 \]
 By the lower semicontinuity of the function $\Lsc(X, \alpha_{i,j})(g)$ there exists a neighbourhood $V$ of $x$ such that $\alpha_{i,j}(c_x)\ll\Lsc(X, \alpha_{i,j})(g)(y)$, for all $y\in V$. Since $h(y)\le c_x$ for all $y\in U$ it follows that
 \begin{align*}\label{eq: inj}
 \Lsc(X,\alpha_{i,j})(h)(y)=\alpha_{i,j}(h(y))\le \alpha_{i,j}(c_x)\ll \Lsc(X, \alpha_{i,j})(g)(y),
 \end{align*}
 for all $y\in U\cap V$.
 
 We have shown that for each $x\in X$ there exist $j\ge i$ and a neighbourhood $W$ of $X$ such that $\Lsc(X,\alpha_{i,j})(h)(y)\le \Lsc(X, \alpha_{i,j})(g)(y)$, for all $y\in W$. Therefore, by the compactness of $X$ we may choose $j\ge i$ such that $\Lsc(X,\alpha_{i,j})(h)(y)\le \Lsc(X, \alpha_{i,j})(g)(y)$, for all $y\in X$. It follows now that $\Lsc(X,\alpha_{i,j})(h)\le \Lsc(X, \alpha_{i,j})(g)$.
This proves condition (b).
 
Observe that condition (a) above is equivalent to saying that $\bigcup_i\Lsc(X,\alpha_{i})(\Lsc(X,M_i))$
forms a dense subset in $\Lsc(X,M)$.  
Let $g_1,g_2\in \Lsc(X,M)$ be such that $g_1\ll g_2$. By Proposition~\ref{prop:updirected} (and its proof), there exists a piecewise characteristic function $h=\chi(\mathcal U,\varphi)$ such that $g_1\ll h\ll g_2$ whose range can be chosen in a dense subset of $M$.

Since $\bigcup_{i\geq 1}\alpha_{i}(M_i)$ forms a dense subset of $M$, and a piecewise characteristic function
takes a finite number of values in $M$, we can find  $k\geq 1$  such that furthermore $\varphi(\mathcal A_{\mathcal U})\subseteq \alpha_{k}(M_k)$.
For each $F\in \mathcal A_{\mathcal U}$ let us write $\varphi(F)=\alpha_{k}(c_{F})$ for some $c_F\in M_k$. Recall that $\varphi$ is an ordered map
and, in fact, by the proof of Proposition~\ref{prop:updirected}, we actually have $\alpha_{k}(c_F)\ll \alpha_{k}(c_{F'})$ whenever $F\subsetneq F'$.
Let us write each $c_F$ as the supremum of a rapidly increasing sequence in $M_k$, $c_F=\sup_i c_F^i$. Hence, for each 
$N\geq 0$ there exists $i_N$ such that $i_N\geq N$ and $\alpha_k(c_F)\ll \alpha_k(c_{F'}^{i_N})$ whenever $F\subsetneq F'$. Now, using that $M=\varinjlim_i M_i$ and
since $\mathcal A_{\mathcal U}$ has a finite number of inequalities, we can choose $l_N\geq k$
such that $\alpha_{k,l_N}(c^{i_N}_F)\leq \alpha_{k,l_N}(c^{i_N}_{F'})$ for all $F\subsetneq F'$. 
Therefore $\varphi_{N}(F):=\alpha_{k,{l_N}}(c^{i_N}_F)$ defines an ordered map $\varphi_N\colon\mathcal A_{\mathcal U}\to M_{l_N}$, and 
$h_{N}:=\chi(\mathcal U,\varphi_{N})$ is a piecewise characteristic function in $\Lsc(X,M_{l_N})$. It follows that $h=\sup_N \Lsc(X,\alpha_{l_N})(h_{N})$. 
Since $g_1\ll h \ll g_2$
there exists $N_0$ such that $g_1\ll \Lsc(X,\alpha_{l_{N_0}})(h_{N_0})\ll g_2$, proving that 
the images $\Lsc(X,\alpha_{i})(\Lsc(X,M_i))$ form a dense subset in $\Lsc(X,M)$.
\end{proof}

\section*{Acknowledgements}

This work has been partially supported by a
MEC-DGESIC grant (Spain) through Project MTM2008-0621-C02-01/MTM,
and by the Comissionat per Universitats i Recerca de la Generalitat
de Catalunya.


\end{document}